\documentclass[11pt]{amsart}

\usepackage[utf8]{inputenc}
\usepackage{amsmath, amssymb,amsthm,tikz}
\usepackage{tikz-cd}
\usepackage{tikz}
\usepackage{todonotes}
\usepackage{mathrsfs}
\usepackage{extarrows}
\usepackage{graphicx}
\usepackage{thmtools}
\usepackage{mathtools}
\usepackage{hyperref}
\usepackage{multirow}
\usepackage{esvect} 

\usepackage{nicematrix}
\usepackage{adjustbox}
\usepackage[margin=1.48in]{geometry}

\usepackage[english]{babel}
\usepackage{comment}

\usepackage[toc,page]{appendix}

\newcommand{\ZZ}{\mathbb{Z}}
\newcommand{\RR}{\mathbb{R}}
\newcommand{\CC}{\mathbb{C}}
\newcommand{\QQ}{\mathbb{Q}}
\newcommand{\Qlbar}{\overline{\mathbb{Q}}_\ell}
\newcommand{\FF}{\mathbb{F}}

\newcommand{\Fq}{{\mathbb{F}_q}}
\newcommand{\Fp}{{\mathbb{F}_p}}
\newcommand{\cA}{\mathcal{A}}
\newcommand{\cC}{\mathcal{C}}
\newcommand{\cD}{\mathcal{D}}

\newcommand{\cB}{\mathcal{B}}
\newcommand{\mcD}{\mathcal{D}}

\newcommand{\bbG}{\mathbb{G}}
\newcommand{\bbL}{\mathbb{L}}

\newcommand{\btau}{{\boldsymbol{\tau}}}

\newcommand{\bbbL}{{\underline{\bbL}}}
\newcommand{\bbbT}{{\underline{\bbT}}}
\newcommand{\bbbU}{{\underline{\bbU}}}
\newcommand{\bbS}{\mathbb{S}}
\newcommand{\bbT}{\mathbb{T}}
\newcommand{\bbU}{\mathbb{U}}

\newcommand{\LLC}{\mathrm{LLC}}

\newcommand{\cE}{\mathcal{E}}
\newcommand{\cG}{\mathcal{G}}
\newcommand{\cL}{\mathcal{L}}
\newcommand{\cM}{\mathcal{M}}

\newcommand{\cP}{\mathcal{P}}
\newcommand{\cT}{\mathcal{T}}
\newcommand{\cU}{\mathcal{U}}
\newcommand{\cX}{\mathcal{X}}
\newcommand{\cY}{\mathcal{Y}}

\newcommand{\End}{\operatorname{End}}
\newcommand{\sgn}{\operatorname{sgn}}

\newcommand{\ind}{\operatorname{ind}}

\newcommand{\Frob}{\mathrm{Fr}}

\newcommand{\bG}{\mathbf{G}}
\newcommand{\bL}{\mathbf{L}}

\newcommand{\bP}{\mathbf{P}}
\newcommand{\bS}{\mathbf{S}}
\newcommand{\bT}{\mathbf{T}}
\newcommand{\bU}{\mathbf{U}}

\newcommand{\ad}{\mathrm{ad}}

\newcommand{\pr}{\operatorname{pr}}

\newcommand{\ram}{\operatorname{ram}}

\newcommand{\sconn}{\mathrm{sc}}

\newcommand{\unr}{\operatorname{ur}}

\newcommand{\ord}{\operatorname{ord}}

\newcommand{\Aut}{\operatorname{Aut}}

\newcommand{\Hom}{\operatorname{Hom}}

\newcommand{\Res}{\operatorname{Res}}

\newcommand{\ii}{\operatorname{i}}

\newcommand{\cInd}{\operatorname{c-Ind}}

\newcommand{\Nor}{\operatorname{N}}
\newcommand{\fR}{\mathfrak{R}}
\newcommand{\fB}{\mathfrak{B}}
\newcommand{\fgg}{\mathfrak{g}}

\newcommand{\Irr}{\operatorname{Irr}}
\newcommand{\Stab}{\operatorname{Stab}}

\newcommand{\fL}{\mathfrak{L}}
\newcommand{\fo}{\mathfrak{o}}
\newcommand{\fp}{\mathfrak{p}}

\newcommand{\fs}{\mathfrak{s}}
\newcommand{\fS}{\mathfrak{S}}
\newcommand{\ft}{\mathfrak{t}}
\newcommand{\lft}{\mathrm{lft}}
\newcommand{\fft}{\mathrm{ft}}
\newcommand{\bound}{\mathrm{b}}

\newcommand{\fX}{{\mathfrak{X}}}
\newcommand{\fU}{\mathfrak{U}}
\newcommand{\abel}{{\mathrm{ab}}}

\newcommand{\der}{{\mathrm{der}}}
\newcommand{\Lie}{{\mathrm{Lie}}}
\newcommand{\rN}{{\mathrm{N}}}
\newcommand{\rP}{{\mathrm{P}}}
\newcommand{\rQ}{{\mathrm{Q}}}
\newcommand{\nr}{{\mathrm{nr}}}
\newcommand{\red}{{\mathrm{red}}}
\newcommand{\res}{{\mathrm{res}}}
\newcommand{\scn}{{\mathrm{sc}}}
\newcommand{\sep}{{\mathrm{s}}}
\newcommand{\enh}{{\mathrm{e}}}
\newcommand{\depth}{{\mathrm{d}}}
\newcommand{\ordi}{{\mathrm{ord}}}
\newcommand{\bbbG}{{\underline{\mathbb{G}}}}

\newcommand{\bbbS}{{\underline{\mathbb{S}}}}
\newcommand{\rZ}{{\mathrm{Z}}}

\newcommand{\bepsilon}{{\boldsymbol{\epsilon}}}

\newcommand{\Gal}{\operatorname{Gal}}

\newcommand{\GL}{\mathrm{GL}}

\newcommand{\SL}{\operatorname{SL}}

\newcommand{\rG}{\operatorname{G}}

\newcommand{\Sp}{\operatorname{Sp}}

\newcommand{\Mat}{\operatorname{M}}
\newcommand{\asym}{\mathrm{asym}}
\newcommand{\symm}{\mathrm{sym}}

\newcommand{\scusp}{\mathrm{sc}}
\newcommand{\cusp}{\mathrm{cusp}}
\newcommand{\cus}{\mathrm{c}}
\newcommand{\superc}{\mathrm{sc}}
\newcommand{\Sc}{\mathrm{Sc}}
\newcommand{\Inn}{{\mathrm{Inn}}}
\newcommand{\InnT}{{\mathrm{InnT}}}
\newcommand{\isom}{\xrightarrow{\;\sim\;}}

\def\infl{\mathrm{infl}}
\def\rq{\mathrm{q}}

\def\fg{{\mathfrak g}}
\def\fA{{\mathfrak{A}}}
\def\fG{{\mathfrak{G}}}

\def\Unip{\mathfrak{U}}

\def\ord{\mathrm{ord}}
\def\Kal{\mathrm{Ka}}
\def\Scmap{\mathrm{Sc}}
\hypersetup{
    colorlinks,
    citecolor=blue,
    filecolor=blue,
    linkcolor=blue,
    urlcolor=blue
}

\renewcommand{\tilde}{\widetilde}

\title[Bruhat-Tits buildings, representations  and Langlands correspondence]{Bruhat-Tits buildings, representations of $p$-adic groups and Langlands correspondence}

\author{Anne-Marie Aubert}
\address{Sorbonne Universit\'e and Universit\'e Paris Cit\'e, CNRS,
IMJ-PRG, F-75005 Paris, France}
\email{anne-marie.aubert@imj-prg.fr}

\date{\today}

\numberwithin{equation}{subsection}
\newtheorem{theorem}[equation]{Theorem}
\newtheorem{prop}[equation]{Proposition}
\newtheorem{thm}{Theorem}

\newtheorem{lem}[equation]{Lemma}

\theoremstyle{definition}
\newtheorem{defn}[equation]{Definition}
\newtheorem{remark}[equation]{Remark}
\newtheorem{example}[equation]{Example}

\begin{document}

\maketitle

\centerline{\sl To the memory of Jacques Tits with admiration}

\begin{abstract}
The Bruhat-Tits theory is a key ingredient in the construction of irreducible smooth representations of $p$-adic reductive groups. We describe generalizations to arbitrary such representations of several results recently obtained in the case of supercuspidal representations, in particular regarding the local Langlands correspondence and the internal structure of the $L$-packets.

We prove that the enhanced $L$-parameters with semisimple cuspidal support are those which are obtained via the (ordinary) Springer correspondence.  

Let  $\bG$ be a connected reductive group  over a non-archimedean field $F$ of residual characteristic $p$. In the case where $\bG$ splits over a tamely ramified extension of $F$ and $p$ does not divide the order of the Weyl group of $\bG$, we show that the enhanced $L$-parameters with semisimple cuspidal support correspond to the irreducible smooth representations of $\bG(F)$ with non-singular supercuspidal support via the local Langlands correspondence constructed by Kaletha, under the assumption that the latter satisfies certain expected properties. As a consequence, we obtain that every compound $L$-packet of $\bG(F)$ contains at least one representation with non-singular supercuspidal support.\end{abstract}

\tableofcontents

\section{Introduction}
Let $F$ be a non-archimedean local field, let $\bG$ be a connected reductive algebraic group  defined over $F$, and let $G$ denote the group of $F$-points of  $\bG$. 
A fundamental problem is the construction (and classification) of all irreducible smooth complex representations of the group $G$. A remarkable breakthrough in this direction was done by Moy and Prasad via the use of the Bruhat-Tits theory.  In \cite{BTI, BTII}, Bruhat and Tits defined an affine building $\cB(\bG,F)$ attached to $G$ on which the group $G$ acts. In particular, for each point $x$ in  $\cB(\bG,F)$, they constructed a compact subgroup $G_{x,0}$ of $G$, called a parahoric subgroup of $G$, which generalizes the notion of Iwahori subgroup. 

In \cite{Moy-PrasadI, Moy-PrasadII}, Moy and Prasad defined a filtration of these parahoric subgroups by smaller, normal in $G_{x,0}$,
subgroups $G_{x,0} \rhd G_{x,r_1} \rhd G_{x,r_2} \rhd G_{x,r_3}\rhd\cdots$, where $0 < r_1 < r_2 < r_3<\cdots$ are real numbers depending on $x$. These subgroups play a fundamental role in the study and construction of supercuspidal representations of $G$. Thanks to this filtration, they introduced  the notion of depth of a representation, which measures the first occurrence of a fixed vector in a given representation.

In \cite{Adler}, Adler used the Moy-Prasad filtration to construct certain positive-depth supercuspidal representations of $G$, in the case where $\bG$  splits over a tamely ramified extension.  His construction  was generalized by Yu in \cite{Yu}. When $\bG$  splits over a tamely ramified extension and the residual characteristic of $F$ does not divide the order of the Weyl group of $\bG$, Yu's construction gives all supercuspidal irreducible representations of $G$ (see \cite{Fintzen,Kim}). However, it was recently noticed by Spice that the proofs of two essential results in this construction, \cite[Prop.~14.1 and Th.~14.2]{Yu}, are not correct, due to the usage of a misstated lemma in a reference. Although  \cite[Th.~3.1]{Fintzen-correction} shows that Yu’s construction still produces irreducible supercuspidal representations, it appeared to be important to restore the validity of Yu's results. This was recently done in \cite[Th.~4.1.13]{FKS} by modifying Yu’s construction with a certain quadratic character.

The original Yu's method is based on the construction of very specific pairs formed by an open compact subgroup of $G$ and an irreducible smooth representation of the latter, called \textit{types} (see \cite{BKtyp}). Yu's construction of types has been generalized to not necessarily supercuspidal representations in \cite{Kim-Yu} and \cite{Fintzen}. In \S\ref{sec:types}, we
similarly extend the ``twisted Yu's construction'' of \cite{FKS}  to arbitrary irreducible smooth representations.

If  the group $\bG$ splits over a tamely ramified  extension and the residual characteristic of $F$ does not divide the order of the Weyl group of $\bG$,  for a certain class of supercuspidal irreducible representations, called ``non-singular'' (which contains the class of ``regular'' supercuspidal representations considered in \cite{Kal-reg}), Kaletha obtained a simplification of Yu's construction and constructed a Langlands correspondence in \cite{Kaletha-nonsingular}. The  $L$-packets which contain non-singular supercuspidal representations of $G$ are those attached to discrete $L$-parameters for $G$ that have trivial restriction to $\SL_2(\CC)$. These $L$-packets contain only supercuspidal representations.  

We generalize these notions to arbitrary smooth representations as follows: we call
an irreducible smooth representation $\pi$ of $G$  \textit{regular} (resp. \textit{non-singular}) if its supercuspidal support is regular (resp. non-singular). We study the local Langlands correspondence for these representations by using the approach we developed in \cite{AMS1,AMS2,AX-Hecke}.
In particular, assuming  some properties of the local Langlands correspondence (proved to be satisfied in the case of classical $p$-adic  groups in \cite{AMS4}, in the case of the group $\rG_2$ in \cite{AX-G2}, and,  for $G$ arbitrary, in the case where $L$ a torus in \cite{ABPS-disc, Sol-ppal}),  we obtain  in Theorem~\ref{thm:main}, for every $L$-parameter $\varphi$,  a description of the equivalence classes of the irreducible non-singular representations in a compound  $L$-packet $\Pi_\varphi$ via the Springer correspondence of a (possibly disconnected) reductive complex group attached to $\varphi$. As a consequence, we prove that, if $G$ is quasi-split, then any $L$-packet for $G$ contains a non-singular representation.

\medskip
\noindent
{\bf Acknowledgements.} We would like to thank the referees for their thourough reports and their
suggestions for improvements. We are also grateful to Alexandre Afgoustidis, Alexander Bertoloni-Meli, Arnaud Mayeux, and Masao Oi for several
helpful comments and discussions.

\bigskip

\noindent
{\bf Notations.} Let $F$ be a non-archimedean local field. Let $\fo_F$ denote the ring of integers of $F$ and  $\fp_F$ the maximal ideal in $\fo_F$. There is in $\fo_F$ a prime element $\varpi_F$ such that $\fp_F=\varpi_F\fo_F$. 
Let $k_F:=\fo_F/\fp_F$ be the residue field of $F$. Then $k_F$ is finite, and we denote by $p$ its characteristic and by $q$ its cardinality. 
We fix a separable closure $F_\sep$ of $F$, and denote by $\Gamma_F:=\Gal(F_\sep/F)$ the absolute Galois group of $F$. We set $\Sigma_F:=\Gamma_F\times\{\pm 1\}$. We denote by $W_F$  the absolute Weil group of $F$ and by $I_F$ (resp $P_F)$  the inertia (resp. wild inertia) subgroup. 
Subextensions of $F_\sep/F$ are always endowed with the valuation extending the valuation $\ordi$ on $F $. We set $|a|_F:=q^{-\ordi(a)}$ for any $a\in F^\times$ and $|0|_F:=0$.  We fix an additive character $\psi\colon F\to\CC^\times$ of $F$ of conductor $\fp_F$.

Let $F_{\unr}$ denote the maximal unramified extension of $F$ within $F_\sep$. The residue field of $F_{\unr}$ is an algebraic closure $\overline k_F$ of $k_F$.  We denote by $\Frob$ the element of $\Gal(F_{\unr}/F)$  that induces the automorphism $\overline a\mapsto \overline a^q$ on  $\overline k_F$. 

Let  $\bG$ a connected reductive algebraic group defined over $F$, and let $G$ be the group of $F$-rational points of $\bG$.  
We denote by $\bG_\ad:=\bG/\rZ_{\bG}$ the adjoint quotient of $\bG$, by  $\bG_\der$ the  derived subgroup of $\bG$, and by $\bG_{\abel}:=\bG/\bG_{\der}$ the maximal abelian quotient of $\bG$. Let $\bG_\scn$ denote the simply-connected cover of $\bG_\der$, and let $\rZ_{\bG}$ be the center of $\bG$. Let $\bL$ be a Levi subgroup of $\bG$. We write $\bL_\scn$ for the pre-image of $\bL$ in $\bG_\scn$, and $\bL_\ad$ for the image of $\bL$ in $\bG_\ad$, and we write $\bL_{\scn,\abel}:=\bL_\scn/\bL_{\scn,\der}$ for the abelianization  of $\bL_\scn$. 

If $\bT$ is a maximal torus in $\bG$, which is defined over $F$. We denote by $X(\bT)$ the group of algebraic group
homomorphisms $\bT\to\bbG_m$ and  by $R(\bG,\bT)\subset X(\bT)$  the root system of $\bG$ with respect to $\bT$. Let $\rQ(\bG,\bT)$ be
the $\ZZ$-sublattice of $X(\bT)$ generated by $R(\bG,\bT)$.  Let $R(\bG,\bT)^\vee$ the set of coroots. We write
\[\rP(\bG,\bT):=\left\{\xi\in X(\bT) \,:\,\langle \xi,\alpha^\vee\rangle\in\ZZ,\;\;\forall \alpha^\vee\in R(\bG,\bT)^\vee\right\}\] for  the lattice of weights.
Following \cite[VI.\S1.9]{Bourbaki}, we call \textit{connection index} the index of the root lattice in the weight lattice, that is the cardinality of  $\rP(\bG,\bT)/\rQ(\bG,\bT)$.  Recall that a prime is said to be \textit{good} (or \textit{not bad}) \textit{for $\bG$}  if it does not divide the coefficients of the highest
root of $R(\bG,\bT)$ (see \cite[\S4.3]{SpSt}). If a good prime for $\bG$ does not divide the connection index, it is said to be \textit{very good for $\bG$} (see for instance \cite[\S1]{Letellier}).

If $\bL$ is a connected reductive $F$-subgroup of $\bG$ containing $\bT$, we set 
\begin{equation} \label{eqn:RR}
R(\bG|\bL,\bT):=R(\bG,\bT)\backslash R(\bL,\bT).
\end{equation} 
The sets $R(\bG,\bT)$, $R(\bL,\bT)$ and $R(\bG|\bL,\bT)$ are equipped with a $\Gamma_F$-action.  

For each $\alpha\in R(\bG,\bT)$, we denote by $\Gamma_\alpha$ (resp. $\Gamma_{\pm\alpha}$) the  stabilizer of $\alpha$ (resp. $\{\pm\alpha\}$) in $\Gamma_F$. Let $F_{\alpha}$ (resp. $F_{{\pm\alpha}}$) be the subfield of $F_\sep$ fixed by $\Gamma_\alpha$ (resp. $\Gamma_{\pm\alpha}$):
\begin{equation}  \label{eqn:fields}
F\subset F_{\pm \alpha}\subset F_{\alpha}\quad\longleftrightarrow\quad\Gamma_F\supset \Gamma_{{\pm \alpha}}\supset \Gamma_{{\alpha}}.
\end{equation}
We recall the following terminology of \cite{AS}: 
\begin{itemize}
\item
When $F_{\alpha}=F_{\pm\alpha}$, we say that $\alpha$ is an asymmetric root.
\item
When $F_{\alpha}\ne F_{\pm\alpha}$, we say that $\alpha$  is a symmetric root. Note that, in this case,
the extension $F_{\alpha}/F_{\pm\alpha}$ is necessarily quadratic. Furthermore,
\begin{itemize}
\item[--]
when $F_{\alpha}/F_{\pm\alpha}$ is unramified, we say that $\alpha$ is symmetric unramified, 
\item[--]
when $F_{\alpha}/F_{\pm\alpha}$ is ramified, we say $\alpha$ is symmetric ramified.
\end{itemize}
\end{itemize}

Let $k_\alpha$ (resp. $k_{\pm\alpha}$) denote the residue fiel of $F_{\alpha}$ (resp. $F_{\pm\alpha}$), let $\rN_{k_\alpha|k_F}\colon k_\alpha^\times\to k_F^\times$ be the norm map, and let $k_\alpha^1$ be the kernel of $\rN_{k_\alpha|k_F}$.
If $p$ is odd, the groups $k_\alpha^\times$ and $k_\alpha^1$ are finite cyclic of even order and we denote by $\sgn_{k_\alpha}$ and $\sgn_{k_\alpha^1}$ the corresponding unique non-trivial $\{\pm 1\}$-valued characters.

For any maximal $F$-split torus $\bT$, $R(\bG,\bT,F)$ denotes the corresponding set of roots in $\bG$. For $\alpha\in R(\bG,\bT,F)$, let $\bU_\alpha$ be the root subgroup corresponding to $\alpha$.

From now on, we suppose that the group $\bG$ splits over a tamely ramified extension of $F$, and that $p$ is very good for $\bG$.

\section{Bruhat-Tits buildings} \label{sec:buildings}
In  \cite{BTI,BTII}, Bruhat and Tits defined a building  associated to  the $p$-adic group $G=\bG(F)$ on which the group $G$ acts. Let $\cB(\bG/\rZ_{\bG},F)$ denote the Bruhat and Tits building of $(\bG/\rZ_{\bG})(F)$. The enlarged building of $G$ is defined to be
\begin{equation}
\cB(\bG,F):=\cB(\bG/\rZ_{\bG},F)\times X_*(\rZ_{G})\otimes_\ZZ\RR,
\end{equation}
where $X_*(\rZ_G)$ is the set of $F$-algebraic cocharacters of $\rZ_G$. 
For each point $x$ in $\cB(\bG,F)$, they constructed a compact subgroup $G_{x,0}$ of $G$, that they called a \textit{parahoric subgroup}. 
Let $G_{x,0^+}$ denote the pro-$p$ unipotent radical of $G_{x,0}$, and $\bbG_{x,0}$ the quotient $G_{x,0}/G_{x,0+}$, (the points of) a connected reductive group over the residue field of $F$. We denote by $[x]$ the image of $x$ in the reduced Bruhat–Tits building and we write $G_{[x]}$ for the stabilizer of $[x$] under the action $G$ on the reduced  building. The group $G_{[x]}$ coincides with the normalizer in $G$  of $G_{x,0}$. Indeed, if $gG_{x,0}g^{-1}$ (which is equal to $G_{g\cdot x,0}$)  is equal to $G_{x,0}$, then we have $g.[x]=[x]$ by \cite[proof of 5.2.8]{BTII}. 

As shown in \cite[\S5]{BTII}, the point $x$ specifies a smooth connected affine $\fo_F$-group scheme $\fG^\circ_x$ with $\fG_x^\circ(F)=\bG(F)$ and $\fG_x^\circ(\fo_F) = G_{x,0}$. We denote by $\bbbG_{x,0}$ the reductive quotient of the special fiber of $\fG^\circ_x$. We have $\bbbG_{x,0}(k_F)=\bbG_{x,0}$\footnote{\textcolor{red}{V\'erifier!}}.
We further have an $\fo_F$-group scheme $\fG_x$ with $\fG_x(F)=G$ and $\fG_x(\fo_F)=\Stab_{G_1}(x)$, 
where $G_1$ denotes the intersection of the kernels of the group homomorphisms $\ordi\circ\chi\colon G\to\ZZ$ for all $F$-rational characters $\chi\colon G\to\bbG_m$. We  denote by $\bbbG_x,$ the quotient of the special fiber of this group scheme by its maximal connected normal unipotent subgroup. Then $\fG^\circ_x$ coincides with the neutral connected component of $\fG_x$. In particular, $\bbbG_x$ is a (possibly disconnected) algebraic group over $k_F$ with reductive neutral connected component. We have $\bbbG_x(k_F)=\bbG_x$.

In \cite{Moy-PrasadI,Moy-PrasadII}, Moy and Prasad defined a filtration of these parahoric subgroups by normal in $G_{x,0}$ subgroups
\begin{equation}
G_{x,0} \rhd G_{x,r_1} \rhd G_{x,r_2} \rhd G_{x,r_3}\rhd\cdots,
\end{equation}
that are normal inside each other and whose intersection is trivial, where $0 < r_1 < r_2 < r_3<\cdots$ are real numbers depending on $x$. For $r\in\RR_{\ge 0}$, we write \begin{equation}
G_{x,r+}:= \bigcup_{t>r} G_{x,t}.
\end{equation} 
When $r>0$, the quotient $G_{x,r}/G_{x,r+}$ is abelian and can be identified with a $k_F$-vector space.
Moy and Prasad defined also filtration submodules $\fg_{x,r}$  of $\fg:=\Lie(\bG)(F)$  and $\fg^*_{x,r}$ 
of the linear dual $\fg^*$ of $\fg$, for $r\in\RR$. We write $\fg_{x,r+}:=\bigcup_{t>r}\fg_{x,t}$. 

The $\fg_{x,r}$ were used in \cite{Moy-PrasadI} to define the notion of depth of a representation. 
We recall it below, following  \cite{DB} in which the reliance on $\fg$ has been removed from the definition. Let $(\pi,V )$ be a smooth representation of $G$. By \cite[Lemma~5.2.1]{DB}, there exists  $\depth(\pi)\in\QQ$ with the following properties.
\begin{enumerate} 
\item[(1)] If $(x,r) \in \cB(\bG,F) \times\RR_{\ge 0}$ such that $V^{G_{x,r+}}$ is nontrivial, then $r\ge \depth(\pi)$.
\item[(2)] There exists a $x'\in\cB(\bG,F)$ such that $V^{G_{x',\depth(\pi)+}}$ is nontrivial.
\end{enumerate}
Then the  \textit{depth of $\pi$}  is defined to be the least nonnegative real number $\depth(\pi)$ for which there exists an $x\in \cB(\bG,F)$ 
such that $V^{G_{x,\depth(\pi)+}}$ is nontrivial.

Note that in \cite{Yu-models}, different filtrations where introduced. It could be possible to define a notion of depth using these filtrations instead. However, as shown in \cite[\S8.1]{Yu-models} or \cite[\S13]{KP}, in the case of quasi-split groups that split over a tamely ramified extension, both notions of filtrations coincide. Related filtrations by affinoid groups, in the Berkovich analytification of a connected reductive group, have been defined in \cite{Ma} (see also \cite{RTW} and the references there).

Let $\bS$ be a torus of $\bG$ defined over $F$. The topological group $S:=\bS(F)$ has a unique maximal bounded subgroup $S_\bound$ (which is also the unique maximal compact subgroup, as $F$ is locally compact) and this subgroup is equipped with a decreasing filtration $S_r$ indexed by the non-negative real numbers, namely the Moy-Prasad filtration corresponding to the unique point in the reduced Bruhat-Tits building of $\bS$. The torus $\bS$ possesses an lft-N\'eron model $\fS^{\lft}$ by \cite[\S10]{BLR}. This is a smooth group scheme over $\fo_F$ satisfying a certain universal property. It is locally of finite type and the maximal subgroup-scheme of finite type is called the ft-N\'eron model $\fS^\fft$. 
Both models share the same neutral connected component $\fS^\circ$, called the connected N\'eron model of $\bS$. We have 
\begin{equation} \label{eqn:S0}
S_0=\fS^\circ(\fo_F)\quad\text{and}\quad S_\bound=\fS^\fft(\fo_F).
\end{equation}
The quotient $S/S_\bound$ is a finitely generated free abelian group, and, by \cite[\S3.2]{Moy-PrasadII}, \cite[\S4.2]{Yu-models}, 
for each positive real number $r$ we have 
\begin{equation} \label{eqn:Sr}
S_r= \left\{s \in S_0\,:\, \ordi(\chi(s)-1)\ge r\text{ for all $\chi\in X^*(\bS)$}\right\}.
\end{equation}
We set $S_{r+}:=\bigcup_{t>r}S_t$. The group $S_{0+}$ is the pro-$p$-Sylow of $S_0$. Let $\bT$ be the maximal unramified subtorus of $\bS$.  By \cite[Lem.3.1.16]{Kal-reg}, we have 
\begin{equation} \label{eqn:bbT}
\bbT:=S_0/S_{0+}\simeq T_0/T_{0+}.
\end{equation}

Let $E/F$ be a finite extension. By a \textit{twisted $E$-Levi subgroup of $\bG$}, we mean an $F$-subgroup $\bG'$ of $\bG$ such that $\bG'\otimes_F E$ is an $E$-rational  Levi subgroup  of an $E$-rational parabolic subgroup of $\bG\otimes_F E$. If $E/F$ is tamely ramified, then $\bG'$ is called a \textit{tamely ramified twisted Levi subgroup of $\bG$}. 
Recall that when $\bG'$ is a tamely ramified twisted Levi subgroup of $\bG$, there is a family of natural embeddings of $\cB(\bG',F)$ into $\cB(\bG,F)$.

Let  $\bG'$ be a tamely ramified twisted Levi subgroup of $\bG$, and denote by $(\Lie^*(\bG'))^{\bG'}$ the subscheme of $\Lie^*(\bG')$ that is fixed by the dual of the
adjoint action of $\bG'$. Let $Y\in(\Lie^*(\bG'))^{\bG'}$. Recall that $(\Lie^*\bT )\otimes F_\sep$ can be identified with $X(\bT)\otimes_\ZZ F_\sep$.  Therefore, we can view $\varpi_r  Y\in\Lie^*(\rZ_{\bG'}\otimes F_\sep)^\circ\subset \Lie^*(\bT \otimes F_\sep)$ as an element of $X(\bT)\otimes_\ZZ \fo_{F_\sep}$, where $\varpi_r$  is an element of $F_\sep^\times$ of valuation $r$. The residue class $\overline Y$ of $\varpi_r Y$ is an element of $X(\bT)\otimes_\ZZ k_{F_\sep}$. Then $\overline Y$ is well defined up to a multiplicative constant in $k_{F_\sep}^\times$.

Following \cite[\S8]{Yu}, as corrected in \cite[Rem.4.1.3]{FKS}, as in \cite[Def.~2.1]{Fintzen-Michigan},  we say that an element $Y$ of $(\Lie^*(\bG'))^{\bG'}$  is \textit{$\bG$-generic of depth $r\in\RR$}  if the following conditions hold:
\begin{enumerate}
\item for some  (equivalently, any) point $x\in\cB(\bG',F)$, we have $Y\in\Lie^*(\bG')_{x,-r}$;
\item $\ordi(Y(H_\alpha))=-r$ for all $\alpha\in R(\bG|\bG',\bT)$ for
some maximal torus $\bT$ of $\bG'$, where $H_\alpha:=\Lie(\alpha^\vee)(1)$ with $\alpha^\vee$ the coroot of $\alpha$.
\item the subgroup of $W$ fixing $\overline Y$ is precisely the Weyl group of $R(\bG',T,F_\sep)$.
\end{enumerate}
As observed in \cite[Lem.~8.1] {Yu}, (2) implies (3) if $p$ is not a torsion prime for the dual root datum of $\bG$.

Let $x\in\cB(\bG',F)$ and $r\in\RR_{>0}$. A character $\phi$ of $G'$ is called \textit{$\bG$-generic}  of depth $r$ relative to $x$ if $\phi$ is trivial on $G'_{x,r+}$, non-trivial on $G'_{x,r}$, and the restriction of $\phi$ to $G'_{x,r}/G'_{x,r+}\simeq\fg_{x,r}' /\fg'_{x,r+}$ is given by $\phi\circ Y$ for some generic element $Y\in(\Lie^*(\bG'))^{\bG'}$ of depth $r$.

Let $\bL$ be an $F$-rational Levi subgroup of an $F$-rational parabolic subgroup of $\bG$. 
Let $\bT$ be an elliptic maximally unramified maximal torus of $\bL$, and let $y\in\cB(\bL,F)$ be the point associated to $\bS$. It is a vertex \cite[Lem.3.4.3]{Kal-reg}. 
If $\alpha\in R(\bL,\bT)$, we set
\begin{equation} \label{eqn:ordy}
\ordi_y(\alpha):=\left\{t\in\RR\,:\,\fgg_\alpha(F_\alpha)_{y,t+}\subsetneq\fgg_\alpha(F_\alpha)_{y,t}\right\}.
\end{equation}
An element $X\in \Lie(\bL_{\scn,\abel})^*(F)\subset \Lie(\bL_\scn)^*(F)$ is called \textit{$\bG$-good} if there exists $r\in\RR$ such that for all tamely ramified maximal tori $\bT \subset\bL$ and all $\alpha\in R(\bG|\bL,\bT)$ we have $\ordi(\langle X,H_\alpha\rangle)=-r$.

Let  $\vv \bG=(\bG^0\subset\bG^1\subset\cdots\subset\bG^d=\bG)$ be a \textit{tamely ramified twisted Levi sequence in $\bG$}, that is,  a finite sequence  of twisted $E$-Levi subgroups of $\bG$, with $E/F$ tamely ramified. 
Let $\bL^0$ be a Levi subgroup of $\bG^0$. We denote by $A_{\bL^0}$ the maximal $F$-split torus of the center $\rZ_{\bL^0}$ of $L^0$. For each $i\in\{1,\ldots,d\}$, we define $\bL^i$ to be  the centralizer of $A_{\bL^0}$ in $\bG^i$, it is a Levi subgroup of $\bG^i$. The 
sequence  $\vv \bL(\vv \bG):=(\bL^0,\cdots,\bL^d)$ is a tamely ramified generalized twisted Levi sequence in $\bL$ and $\rZ_{\bL^0}/\rZ_{\bL^d}$ is $F$-anisotropic (see \cite[Lem.~2.4]{Kim-Yu}).

Folowing \cite[Def.~3.2]{Kim-Yu}, we say that an embedding $\iota\colon\cB(\bL,F)\hookrightarrow\cB(\bG,F)$  is \textit{$(y,r)$-generic} if $\bU_{\alpha,\iota(y),r}=\bU_{\alpha,\iota(y),r+}$ for all $\alpha\in R(\bG|\bL,\bS,F)$ where $\bS$ is any maximal $F$-split torus of $\bL$ such that $y\in\cA(\bL,\bT,F)$. Here $\cA(\bL,\bT,F)$ is the apartment associated to $\bS$ in $\cB(\bL,F)$ and $(\bU_{\alpha,\iota(y),r})_{r\in\RR}$ is the filtration on the root group $\bU_\alpha$, $\alpha\in R(\bG,\bS,F)$ so that
$\bU_{\alpha,\iota(y),r}=\bU_\alpha\cap\bG_{\iota(y),r}$, where $(\bG_{\iota(y),r})_{r\in\RR}$ is the Moy-Prasad filtration.

Let $\{\iota\}$ be a commutative diagram of embeddings:
\begin{equation}
\begin{tikzcd}
\cB(\bG^0,F)\arrow[hookrightarrow]{r}{\iota}&\cB(\bG^1,F)\arrow[hookrightarrow]{r}{\iota}&\cdots\arrow[hookrightarrow]{r}{\iota}&\cB(\bG^d,F)\\
\cB(\bL^0,F)\arrow[hookrightarrow]{r}{\iota} \arrow[hookrightarrow]{u}{\iota}& \cB(\bL^1,F)\arrow[hookrightarrow]{r}{\iota}\arrow[hookrightarrow]{u}{\iota}&\cdots\arrow[hookrightarrow]{r}{\iota}&\cB(\bL^d,F))\arrow[hookrightarrow]{u}{\iota}
\end{tikzcd}.
\end{equation}
We denote by $\iota$ any composite embedding in this diagram from $\cB(\bL^i,F)$ to $\cB(\bL^j,F)$ or $\cB(\bG^j,F)$, and from $\cB(\bG^i,F)$ to $\cB(\bG^j,F)$ for $i\le j$. 

The image of a point $y\in \cB(\bL^0,F)$ in $\cB(\bL^i,F)$ is uniquely determined, modulo the translation of $A_{\bL^0}$.
To specify a diagram of embeddings as above,  we need to give  the image  of $y$ in $\cB(\bG^i,F)$ for $0\le i\le d$.
We have the disjoint-union decomposition 
\begin{equation}
    R(\bL^i|\bL^{i-1},\bT)=R(\bL^i|\bL^{i-1},\bT)_\asym\sqcup R(\bL^i|\bL^{i-1},\bT)_\symm
\end{equation}
into the subsets of asymmetric, respectively symmetric, elements.
We suppose there exists and fix a character $ \phi$ of $L^0$ that is $\bL$-generic relative to $y$ of depth $r$ for some non-negative real number $r$ in the sense of \cite[p.599]{Yu}, and put $s=r/2$. To the character $ \phi$ one can associate a generic element of the dual Lie algebra of $L^0$. 

For $\alpha\in R(\bL^i|\bL^{i-1},\bT)$, denote by $\alpha_0$ the restriction of $\alpha$ to $\rZ_{\bL^0}$, and by $e_{\alpha}$ (resp. $e_{\alpha_0}$)
the ramification degree of the extensions $F_\alpha/F$ (resp. $F_{\alpha_0}/F$). We set $e(\alpha/\alpha_0):=e_\alpha/e_{\alpha_0}$. Let $\ell_{p'}(\alpha^\vee)$ denote the prime-to-$p$ part of the normalized square length of $\alpha^\vee$, and let $f_{(\bL,\bT)}(\alpha)$ be the toral invariant from \cite[\S4.1]{Kaletha-epipelagic}. Let $\Sigma_F$ act on $F_\sep$ with $\{\pm 1\}$ acting trivially. 
For $t\in\bbT$, we write
\[
\epsilon_{\natural,y}^{\bL^i|\bL^{i-1}}(t):=\prod_{\begin{smallmatrix}\alpha\in R(\bL^i|\bL^{i-1},\bT)_\asym/\Sigma_F\cr s\in\ordi_y(\alpha)\end{smallmatrix}}\sgn_{k_\alpha}(\alpha(t))\,\cdot\,
\prod_{\begin{smallmatrix}\alpha\in R(\bL^i|\bL^{i-1},\bT)_{\symm,\unr}/\Gamma_F\cr s\in\ordi_y(\alpha)\end{smallmatrix}}\sgn_{k_\alpha^1}(\alpha(t));
\]
\[
\epsilon_{\flat,0}^{\bL^i|\bL^{i-1}}(t):=\prod_{\begin{smallmatrix}\alpha\in R(\bL^i|\bL^{i-1},\bT)_\asym/\Sigma_F\cr
\alpha_0\in R(\bL^i|\bL^{i-1},\rZ_{\bL^0})_{\symm,\unr}\cr 2\nmid e(\alpha/\alpha_0)
\end{smallmatrix}}\sgn_{k_\alpha}\alpha(t)\,\cdot\,
\prod_{\begin{smallmatrix}\alpha\in R(\bL^i|\bL^{i-1},\bT)_{\symm,\ram}/\Gamma_F\cr
\alpha_0\in R(\bL^i|\bL^{i-1},\rZ_{\bL^0})_{\symm,\unr}\cr 2\nmid e(\alpha|\alpha_0)\end{smallmatrix}}\sgn_{k_\alpha^1}\alpha(t);
\]
\[
\epsilon_{\flat,1}^{\bL^i|\bL^{i-1}}(t):=\prod_{\begin{smallmatrix}
\alpha\in R(\bL^i|\bL^{i-1},\bT)_{\symm,\ram}/\Gamma_F\cr
\alpha(t)\in -1+\fp_\alpha\end{smallmatrix}}(-1)^{[k_\alpha:k_F]+1}\sgn_{k_\alpha}(e_\alpha\ell_{p'}(\alpha^\vee));
\]
\[
\epsilon_{\flat,2}^{\bL^i|\bL^{i-1}}(t):=\prod_{\begin{smallmatrix}
\alpha\in R(\bL^i|\bL^{i-1},\bT)_{\symm,\ram}/\Gamma_F\cr
\alpha(t)\in -1+\fp_\alpha\end{smallmatrix}}\sgn_{k_\alpha}((-1)^{\frac{(e(\alpha/\alpha_0)-1)}{2}});
\]
\begin{equation}\epsilon_{\flat,0}^{\bL^i|\bL^{i-1}}:=\epsilon_{\flat,0}^{\bL^i|\bL^{i-1}}\cdot \epsilon_{\flat,1}^{\bL^i|\bL^{i-1}}\cdot\epsilon_{\flat,2}^{\bL^i|\bL^{i-1}};
\end{equation}
\begin{equation}
\epsilon_{f}^{\bL^i|\bL^{i-1}}(t):=\prod_{\begin{smallmatrix}
\alpha\in R(\bL^i|\bL^{i-1},\bT)_{\symm,\ram}/\Gamma_F\cr
\alpha(t)\in -1+\fp_\alpha\end{smallmatrix}}f_{(\bL,\bT)}(\alpha).
\end{equation}

By \cite[Th.~3.4 and Cor.~3.6]{FKS}, there exists a unique character 
\begin{equation} \label{eqn:eps-first}
\epsilon_y^{\bL^i|\bL^{i-1}}\colon \bbL^{i-1}_y\to\{\pm 1\}
\end{equation} 
satisfying the following
property: for every tamely ramified maximal torus $\bT$ of $\bL^{i-1}$ such that $y\in \cB(\bT,F)$ the restriction of $\epsilon_y^{\bL^i|\bL^{i-1}}$
to $\bbT$ equals $\epsilon_{\natural,y}^{\bL^i|\bL^{i-1}}\cdot\epsilon_\flat^{\bL^i|\bL^{i-1}}\cdot\epsilon_f^{\bL^i|\bL^{i-1}}$.
By inflation we consider $\epsilon_y^{\bL^i|\bL^{i-1}}$ as a character of $L_y^{i-1}$ when convenient.
For every $i\in\{1,\ldots,d\}$, we define
\begin{equation} \label{eqn:sign}
\epsilon_y^i:=\prod_{j=1}^i\epsilon_y^{\bL^j|\bL^{j-1}}, \quad\text{and we write}\quad \epsilon_y:=\epsilon_y^d.
\end{equation}

\section{Theory of types} \label{sec:types}

\subsection{Construction of types} 
\label{subsec:review}

\begin{defn}  
A \textit{depth-zero $\bG$-datum} is a triple $((\bG,\bL),(y,\iota),\btau_L)$ satisfying the following
\begin{itemize}
\item[(1)] $\bG$ is a connected reductive group over $F$, and $\bL$ is a Levi subgroup of $\bG$;
\item[(2)] $y$ is a point in $\cB(\bL,F)$ such that $L_{y,0}$ is a maximal parahoric subgroup of $L$, and $\iota\colon\cB(\bL,F)\hookrightarrow\cB(\bG,F)$ is a $(y,0)$-generic embedding;
\item[(3)]  $\btau_L$ is an irreducible smooth representation of $L_y$ such that $\btau_L|L_{y,0}$ contains the inflation to $L_{y,0}$ of a cuspidal representation of $\bbL_{y,0}$.
\end{itemize}
\end{defn}

\begin{defn} \label{defn:datum}
A \textit{$\bG$-datum} is a $5$-tuple 
\begin{equation} \label{eqn:Gdatum}
\mcD=((\vv \bG,\bL^0),(y,\{\iota\}),\vv r,\btau_{L^0},\vv \phi)
\end{equation}
satisfying the following:

\begin{itemize}
\item[\bf D1.] $\vv \bG=(\bG^0,\bG^1,\cdots, \bG^d)$ is a tamely ramified twisted Levi sequence in $\bG$, and $\bL^0$ a 
Levi subgroup of $\bG^0$. Let $\vv \bL$ be associated to $\vv \bG$ as above;
\item[\bf D2.] $y$ is a point in $\cB(\bL^0,F)$, and $\{\iota\}$ is a commutative diagram of $\vv s$-generic embeddings of buildings relative to $y$, where $\vv s=(0,r_0/2,\cdots,r_{d-1}/2)$;
\item[\bf D3.] $\vv r=(r_0,r_1,\cdots,r_d)$ is a sequence of real numbers satisfying $0<r_0<r_1<\cdots<r_{d-1}\le r_d$ if $d>0$, and $0\le r_0$ if $d=0$;
\item[\bf D4.] $\mcD^0:=((\bG^0,\bL^0),(y,\iota), \btau_{L^0})$ is a depth zero $\bG^0$-datum;
\item[\bf D5.] $\vv \phi=(\phi_0, \phi_1,\cdots, \phi_d)$ is a sequence of quasi-characters, where $ \phi_i$ is a quasi-character of $G^i$ such that $ \phi_i$ is
$G^{i+1}$-generic of depth $r_i$ relative to $y$ for all $y\in\cB(\bG^i,F)$.
\end{itemize}
\end{defn}

For $i\in\{0,\ldots,d\}$, we write $s_i:=r_i/2$. We set $s_{-1}:=0$. For $i\in\{1,\ldots,d\}$ and $t=s_{i-1}$ or $t= s_{i-1}+$, we denote by $(\bG^{i-1},\bG^i)(F)_{\iota(y),(r_{i-1},t)}$ the group
\[G^{i-1}\cap \langle\bT(E)_{r_{i-1}}, \bU_\alpha(E)_{\iota(y),r_{i-1}},\bU_\beta(E)_{\iota(y),t}\;:\;\alpha\in R(\bG^i,\bT), \beta\in R(\bG^{i-1}|\bG^i,\bT)\rangle,\]
where $\bT$ is a maximal torus of $\bG^i$ that splits over a tamely ramified extension $E/F$ with $\iota(y)\in\cB(\bT,F)\subset\cB(\bG^{i},F)$, and $\bU_\alpha(E)_{\iota(y),r_{i-1}}$ denotes the 
Moy-Prasad filtration subgroup of depth $r_{i-1}$ at $y$ of the root group $\bU_\alpha(E)\subset \bG^{i-1}(E)$ corresponding to the root $\alpha$, and similarly for $\bU_\beta(E)_{\iota(y),t}$. We set
\[J^i:=(\bG^{i-1},\bG^i)(F)_{\iota(y),(r_{i-1},s_{i-1})}\quad\text{and}\quad J^i_+:=(\bG^{i-1},\bG^i)(F)_{\iota(y),(r_{i-1},s_{i-1}+)}.\]
\noindent
{\bf The construction.}
Let $\mcD=((\vv \bG,\bL^0),(y,\{\iota\}),\vv r,\btau_{L^0},\vv \phi)$ be  a $\bG$-datum.  Let $K_{G^0}$ denote the group generated by $L^0_y$ and $G_{\iota(y),0}^0$.
For $0\le i\le d$, as in \cite[\S7.4]{Kim-Yu}, we define 
\begin{equation} \label{eqn:Kii}
\begin{cases}
K^i:=K_{G^0}G^1_{\iota(y),s_0}\cdots G^i_{\iota(y),s_{i-1}}\cr K^i_+:=G_{\iota(y),0+}^0G^1_{\iota(y),s_0+}\cdots G^i_{\iota(y),s_{i-1}+}.
\end{cases}
\end{equation}
Inductively, we show that $K^i$ and $K^i_+$ are groups. For $i=0$, we have $K^0:=K_{G^0}$ and $K^0_+:=G_{\iota(y),0+}^0$. For $i>0$, $K_{G^0}G^1_{\iota(y),s_0}\cdots G^{i-1}_{\iota(y),s_{i-2}}$ is a group by our induction hypothesis. It is clearly a subgroup of $G^i_{\iota(y),0}$. Since $G^i_{\iota(y),0}$ normalizes
$K_{G^0}$, $G^1_{\iota(y),s_0}$, $\ldots$, $G^{i-1}_{\iota(y),s_{i-2}}$, we see that  $K_{G^0}G^1_{\iota(y),s_0}\cdots G^i_{\iota(y),s_{i-1}}$ is a group. The proof that $K^i_+$ is a group is analogous. 

By \cite[Prop.~4.3(b)]{Kim-Yu}, we have
\begin{equation} \label{eqn:rqt}
    K_{G^0}/G^0_{\iota(y),0+}\overset{\rq}{\simeq} L^0_y/L_{y,0+}^0.
\end{equation}
We define $\btau_{\mcD^0}$ to be the representation of $K_{G^0}$ obtained by composing the isomorphism \eqref{eqn:rqt} with $\btau_{L^0}$.

Following similar lines as in \cite{Yu,Kim-Yu}, we construct, for $0\le i\le d$,  irreducible representations $\btau_i$ and $\btau_i'$ of $K^i$ inductively as follows. First we set $\btau_0':=\btau_{\mcD^0}$ and $\btau_0:=\btau_0'\otimes \phi_0$.  
Let $i\in\{1,\ldots,d\}$.
We suppose that $\btau_{i-1}$ and $\btau_{i-1}'$ are already constructed, and that $\btau_{i-1}'|_{G^{i-1}_{\iota(y),r_{i-1}}}$ is $\mathbf{1}$-isotypic.
Let $\hat \phi_{i-1}$  be the extension of $ \phi_{i-1}|_{K_{G^0}G^{i-1}_{\iota(y),0}}$ to $K_{G^0}G^{i-1}_{\iota(y),0}G_{y,s_{i-1}+}$  defined as in \cite[\S4]{Yu}.

Since every element of $J^i/J^i_+$ has order dividing $p$, we can view $J^i/J^i_+$ as an $\Fp$-vector space. 
We define a pairing $\langle\,,\,\rangle$ on $J^i/J^i_+$ by
\begin{equation}  \label{eqn:pairing}
\langle j_1,j_2\rangle:=\hat \phi_{i-1}(j_1j_2j_1^{-1}j_2^{-1})=[j_1,j_2].
\end{equation}
By  \cite[Prop.~6.4.44]{BTI}, it is well-defined, since $[J^i,J^i]$ is contained in $J^i_+$, and by \cite[Lem.~11.1]{Yu}, it is non-degenerate. For any $j\in J_+^i$, the kernel of $\hat \phi_{i-1}$ contains $j^p$. Hence the order of every element of $\hat \phi_{i-1}(J^i_+)$
divides $p$. 
Since $\hat \phi_{i-1}(J^i_+)$ is a nontrivial subgroup of $\CC^\times$, this shows that $\hat \phi_{i-1}(J^i_+)$ is isomorphic to $\Fp$. Thus we can view $\langle\,,\,\rangle$ as a non-degenerate symplectic form on the $\Fp$-vector space $V_i:=J^i/J^i_+$. Let $J^i_{+}(\hat \phi_{i-1})$ denote the intersection of $J^i_+$ with the kernel of $\hat \phi_{i-1}$. From \cite[Prop.~11.4]{Yu}  there is a canonical isomorphism
\begin{equation} \label{eqn:Yuiso}
J^i/J^i_{+}(\hat \phi_{i-1})\longrightarrow V_i\times (J^i_+/J^i_{+}(\hat \phi_{i-1}))\simeq H(V_i),
\end{equation}
where $H(V_i)$ is the Heisenberg group defined to be the set $V_i\times \Fp$ with the multiplication $(v,t)(v',t'):=(v+v',t+t'+\frac{1}{2}\langle v,v'\rangle)$. 
By combining \eqref{eqn:Yuiso} with the map $K^{i-1}\to \Sp(V_i)$ induced by the conjugation, we define a map
\begin{equation} \label{eqn:Yumap}
K^{i-1}\ltimes J^i \to (J^i/J^i_{+}(\hat \phi_{i-1}))\rtimes  K^{i-1}\to H(V_i)\rtimes \Sp(V_i). 
\end{equation}
By combining \eqref{eqn:Yumap} with the Weil representation of $H(V_i)\rtimes \Sp(V_i)$ with central character $\hat \phi_{i-1}$, we obtain a representation $\widetilde \phi_{i-1}$ of $K^{i-1}\ltimes J^i$. Let $\infl( \phi_{i-1})$ denote the inflation of $ \phi_{i-1}$ via the map  $K^{i-1}\ltimes J^i \to K^{i-1}$. Then $\infl( \phi_{i-1})\otimes\widetilde \phi_{i-1}$ factors through the map  $K^{i-1}\ltimes J^i\to K^{i-1}J^i=K^i$. Let $\phi'_{i-1}$ denote the representation of $K^i$ whose inflation to $K^{i-1}\ltimes J^i$ is $\infl( \phi_{i-1})\otimes\widetilde \phi_{i-1}$. If $r_{i-1}< r_i$, the restriction of $\phi'_{i-1}$ to $G^i_{\iota(y),r_i}$ is $\mathbf{1}$-isotypic.

Let $\infl(\btau'_{i-1})$ be the inflation of $\btau'_{i-1}$ via the map 
\begin{equation}
K^i=K^{i-1}J^i\longrightarrow K^{i-1}J^i/J^i\simeq K^{i-1}/G^{i-1}_{\iota(y),r_{i-1}}.
\end{equation}
Then we define $\btau'_i:=\infl(\btau_{i-1}')\otimes \phi'_{i-1}$, which is trivial on $G^i_{\iota(y),r_i}$ if $r_{i-1}< r_i$, and $\btau_i:=\btau_i'\otimes \phi_i$.

Recall the tamely ramified twisted Levi sequence $\vv \bL:=\vv \bL(\vv \bG)=(\bL^0,\ldots, \bL^d)$ of $\bL$. We attach to $\vv \bL$ the cuspidal datum for $\bL$:
\begin{equation} \label{eqn:Mdatum}
\mcD_\bL:=(\vv \bL,y,\vv r,\btau_{L^0},\vv \phi).
\end{equation}
For $i\in\{1,\ldots,d\}$, we define 
\begin{equation}
J_L^i:=L^i\cap J^i\quad\text{and}\quad J_{L,+}^i:=L^i\cap J_+^i.
\end{equation}
Let $\bepsilon_y^{\bL^i|\bL^{i-1}}$ denote the character of $L^{i-1}_{[y]}J_L^i$  that is trivial on
 $J_L^i$, and whose restriction to $L^{i-1}_{[y]}$ is $\epsilon_y^{\bL^i_\ad|\bL^{i-1}_\ad}\circ\pr^{i-1}_\ad$, where $\epsilon_y^{\bL^i_\ad|\bL^{i-1}_\ad}$ is the character  defined in \eqref{eqn:eps-first} and $\pr^{i-1}_\ad\colon (L_\ad^{i-1})_{[y]}\to (\bL_\ad^{i-1})_{[y]}/(\bL_\ad^{i-1})_{y,0+}$ is the natural projection.
 We set
\begin{equation}
K^i_L:=L_y^0L^1_{y,s_0}\cdots L^i_{y,s_{i-1}}\quad\text{and}\quad K^i_{L,+}:=L_{y,0+}^0L^1_{y,s_0+}\cdots L^i_{y,s_{i-1}+}.
\end{equation}
For $1\le j\le i$,  we define $\epsilon_y^{K_L^i}$  to be the quadratic character of $K^i_L$ whose restriction to $L^0_{\iota(y)}$  is given by 
$\bepsilon^{\bL^i|\bL^{i-1}}_y|_{L^0_y}$  and that is trivial on $L^1_{y,s_0}\cdots L^i_{y,s_{i-1}}$.
The datum  $\mcD_\bL$ gives a supercuspidal type in $L^i$ as follows.
Let $\widetilde{K}_{L}^i$ denote the normalizer in $L^i$ of $K_L^i$. The group $\widetilde{K}_{L}^i$ is a compact modulo center subgroup of $L$. We set
\begin{equation} \label{eqn:sigma}
\btau_L^i:=\btau^i|_{K_L}\quad\text{and}\quad\sigma_{\mcD_L}^i:=\ind_{\widetilde{K}_L^i}^{L}(\epsilon_y^{K_L^i}\btau^i_L).
\end{equation}
Then $\sigma_{\mcD_L}^i$ is an irreducible supercuspidal representation of $L^i$ (see  \cite[proof of Th.~4.1.13]{FKS}).

\subsection{Types and Bernstein decomposition} \label{subsec:covers}
Let $\bL$ be an $F$-rational Levi subgroup of an $F$-rational parabolic of $\bG$. 
Let $\fX_\nr(L)$ denote the group of unramified characters of $L$. Let $\sigma$ be an irreducible supercuspidal smooth representation of $L$. 
We write:
\begin{itemize}
\item $(L,\sigma)_G$ for the $G$-conjugacy class of the pair $(L,\sigma)$;
\item $\fs:=[L,\sigma]_G$ for the $G$-conjugacy class of the pair $(L,\fX_\nr(L)\cdot\sigma)$. 
\end{itemize}
We denote by $\fB(G)$ the set of such classes $\fs$. We set $\fs_L:=[L,\sigma]_L$.

Let $\fR(G)$ denote the category of smooth representations of $G$, and let 
$\fR^\fs(G)$ be the full subcategory of $\fR(G)$ whose objects are the representations $(\pi,V)$ such that every $G$-subquotient of $\pi$ is equivalent to a subquotient of a parabolically induced representation $\ii_P^G(\sigma')$, where $\ii_P^G$ is the functor of normalized parabolic induction  and $\sigma'\in\fX_\nr(M)\cdot\sigma$. 
By \cite{Bernstein-centre}, the categories $\fR^\fs(G)$ are indecomposable and split the category $\fR(G)$ into a direct product:
\begin{equation} \label{eqn:Bernstein decompositionI}
\fR(G)=\prod_{\fs\in\fB(G)}\fR^\fs(G).
\end{equation}
Every object of $\fR^\fs(G)$ has depth $\depth(\sigma)$ (see \cite[Lem.5.2.4 and Lem.5.2.6]{DB} for instance). Let
$\Irr(G)$ (resp. $\Irr_\scusp(G)$) denote the set of classes of irreducible (resp. irreducible supercuspidal) objects in $\fR(G)$, and by $\Irr^\fs(G)$  the classes of irreducible objects in $\fR^\fs(G)$, i.e. the equivalence classes of irreducible representations whose supercuspidal support lies in $\fs$, and call $\Irr^\fs(G)$ the \textit{Bernstein series} attached to $\fs$. By \eqref{eqn:Bernstein decompositionI}, Bernstein series give a partition of the set $\Irr(G)$ of isomorphism classes of irreducible smooth irreducible representations of $G$:
\begin{equation} \label{eqn:Bernstein decomposition}
\Irr(G)=\bigsqcup_{\fs\in\fB(G)}\Irr^\fs(G).
\end{equation}

Let $\bP=\bL\bU$ be a parabolic subgroup of $\bG$ defined over $F$ with Levi factor $\bL$, and let $\overline\bP=\bL\overline\bU$ be the opposite parabolic subgroup. A compact open subgroup $K$ of $G$ is said to \textit{decompose with respect} to $(U,L,\overline U)$ if $K=(K\cap U)\cdot (K\cap L)\cdot (K\cap \overline U$
Let $K$ (resp. $K_L$) be a compact open subgroup of $G$ (resp. $L$), and $\lambda$ (resp. $\lambda_L$) an irreducible smooth representation of $K$ (resp. $K_L$). The pair $(K,\lambda)$ is called a \textit{$G$-cover} of the pair $(K_L,\lambda_L)$ (see \cite{BKtyp, Blondel})  if for any opposite pair of parabolic subgroups $\bP=\bL\bU$ and $\overline\bP=\bL\overline\bU$ with  with Levi factor $\bL$, we have
\begin{enumerate}
\item $K$ decomposes with respect to $(U,L,\overline U)$;
\item $\lambda|_{K_L}=\lambda_L$ and $K\cap U, K\cap\overline U\subset\ker(\lambda)$;
\item for any smooth representation $V$ of $G$, the natural map from $V$ to its Jacquet module $V_U$ induces an injection on $V^{(K,\lambda)}$ the $(K,\lambda)$-isotypic subspace of $V$.
\end{enumerate}
A pair  $(K,\lambda)$ is called  an \textit{$\fs$-type} for $G$ if   the following property is satisfied
\begin{equation}
\text{$\lambda$ occurs in the restriction of $\pi\in\Irr(G)$ if and only if $\pi\in\fs$.}
\end{equation}
We recall that if $(K_L,\lambda_L)$ is an $\fs_L$-type for $L$, then any $G$-cover of $(K_L,\lambda_L)$ is an $\fs$-type for $G$ by \cite[Th. 8.3]{BKtyp}.

We define $\epsilon_y^{K^i}$  to be the quadratic character of $K^i$  defined as
\begin{equation}
\epsilon^{K^i}:=\epsilon_y^{i} \,\circ\, \rq,
\end{equation}
where  $\rq$ is the isomorphism \eqref{eqn:rqt}.
\begin{theorem}  \label{thm:cover}
For each $i\in\{0,\ldots,d\}$, the pair
$(K_L^i,\epsilon^{K_L^i}\btau_L^i)$ is an $\fs_L$-type for $L^i$, and 
$(K^i,\epsilon^{K^i}\btau^i)$ is a $G^i$-cover of $(K_L^i,\epsilon^{K_L^i}\btau_L^i)$, thus it is an $\fs^i$-type for $\fs^i=[L^i,\sigma_{\mcD_L}^i]_{G^i}$.
\end{theorem}
\begin{proof} The proof  is analogous to that of \cite[Th.~7.5]{Kim-Yu}.  
\end{proof}

\subsection{Regular representations} \label{subsec:reg}
Let $\bbbL$ be a connected reductive group defined over the finite field $k_F$, let $\bbbT$ be a maximal torus of $\bbbL$.
In \cite[Def.~5.15]{Deligne-Lusztig}, Deligne and Lusztig defined two regularity conditions for a character $\theta$ of $\bbT:=\bbbT(k_F)$, which we now recall:
\begin{itemize}
 \item $\theta$ is said to be \textit{in general position} if its stabilizer in $(\Nor_\bbbG(\bbbT)/\bbbT)(k_F)$ is trivial. 
\item $\theta$ is said to be \textit{non-singular} if it is not orthogonal to any coroot.
\end{itemize} 
By \cite[Prop.~5.16]{Deligne-Lusztig}, if the centre of $\bG$ is connected, $\theta$ is non-singular if and only if it is in general position.

Let $\bS$ be a maximal torus of $\bL $ which  defined over $F$, and let $\bT$ denote the maximal unramified subtorus of $\bS$. Let $F'/F$ be an unramified extension splitting $\bT$ and $R_\res(\bL,\bT)$ denote the set of restrictions to $\bT$ of the absolute roots in $R(\bL,\bS)$.
Let $y$ be a vertex in $\cB(\bL/\rZ_{\bL},F)$, and let  $\bbbT$ be an elliptic maximal torus of $\bbbL_y$. By \cite[Lem. 2.3.1]{DB06} there exists a maximally unramified elliptic maximal torus $\bT$ of $\bG$ such that the reductive quotient of the special fiber of the connected N\'eron model of $\bT$ is $\bbbT$.
In \cite[Def.3.4.16]{Kal-reg}, Kaletha introduced the following definition: a character $\theta$ of $\bbbT$ is called \textit{regular} if its stabilizer in $(\Nor_\bL(\bT)/\bT)(F)$ is trivial. By \cite[Fact 3.4.18]{Kal-reg}, if $\theta$ is regular, then it is in general position. 

Let $\vartheta$ be a depth-zero character of $S$. Let $S_0$ and $\bbT$ be defined as in \eqref{eqn:S0} and \eqref{eqn:bbT}. The character $\vartheta$ is  said to be \textit{regular} if its restriction to $S_0$ equals the inflation of a regular character of $\bbT$  (see \cite[Def.~3.4.16]{Kal-reg}).

Let $\sigma$ be any irreducible depth-zero supercuspidal representation of $L$. There exists a vertex $y\in\cB_\red(\bL,F)$ and an irreducible cuspidal representation $\tau$ of $\bbL_y:=\bbbL_y(k_F)$, such that the restriction of $\sigma$ to $L_{y,0}$ contains the inflation $\infl(\tau)$ of $\tau$ (see \cite[\S1-2]{Morris-ENS} or  \cite[Prop.~6.6]{Moy-PrasadII}). The normalizer $\rN_{L}(L_{y,0})$ of $L_{y,0}$ in $L$ is a  totally disconnected group that is compact mod center, which, by \cite[proof of (5.2.8)]{BTII}, coincides with the fixator $L_{[y]}$ of $[y]$ under the action of $L$ on the reduced building of $\bL$. Let $\btau$ denote an extension of $\infl(\tau)$ to $L_{[y]}$. Then $\sigma$ is compactly induced from a representation of $\rN_L(L_{y,0})$: 
\begin{equation} \label{eqn:depth zero supercuspidal}
\sigma=\cInd_{L_{[y]}}^L(\btau).
\end{equation}
The representation $\sigma$ is called \textit{regular} if $\tau=\pm R_\bbbT^{\bbbL_{y,0}}(\theta)$ for some pair $(\bbbT,\theta)$, where $\bbbT$ is an elliptic maximal torus of $\bbbL_{y,0}$ and $\theta$ is a regular character of $\bbT$ (see \cite[Def.~3.4.19]{Kal-reg}).

A supercuspidal representation of $L$ is called \textit{regular} if it arises via the construction described in \S\ref{subsec:review} from a cuspidal $L$-datum $\cD^L$ such that the representation $\sigma^0$ of $L^0$ is regular (see \cite[Def.~3.7.9]{Kal-reg}).

\begin{defn} \label{defn:regular}
An irreducible smooth representation $\pi$ of $G$ is \textit{regular}  if its supercuspidal support is regular, i.e. if $\pi\in\Irr^\fs(G)$ with $\fs=[L,\sigma]_L$ such that $\sigma$ is regular.
\end{defn}

\begin{prop} \label{prop:GL}
Any irreducible smooth representation of $\GL_N(F)$ is regular.
\end{prop}
\begin{proof}
Let $G:=\GL_N(F)$, let $\fs\in\fB(G)$ and let $\pi\in\Irr^\fs(G)$.  There is a partition $(N_1,N_2,\ldots,N_r)$ of $N$ such that $\fs=[L,\sigma]_G$ with
\[L=\GL_{N_1}(F)\times\GL_{N_2}(F)\times\cdots\times\GL_{N_r}(F)\quad\text{and}\quad \sigma=\sigma_1\otimes\sigma_2\otimes\cdots\otimes\sigma_r,\]
with $\fs_i:=[\GL_{N_i}(F),\sigma_i]_{\GL_{N_i(F)}}\in\fB(\GL_{N_i(F)})$.  The representation $\sigma$ is regular if and only if the representations $\sigma_1$, $\ldots$, $\sigma_r$ are all regular. Thus, by Definition~\ref{defn:regular}, the representation $\pi$ is regular if and only if the representations $\sigma_1$, $\ldots$, $\sigma_r$ are all regular.  Thus we are reduced to show that any supercuspidal irreducible smooth representation $\pi$ of $\GL_N(F)$ is regular. Then the result follows from \cite[Lem.3.7.7]{Kal-reg}. We can also recover it as follows. By applying the construction above with $L=G$, we obtain that $G^0\simeq\GL_{N_E}(E)$, where $E/F$ is finite tamely ramified extension (here $N_E=N/[E:F]$) and $\bbG^0_y$ is isomorphic to $\GL_{N_E}(k_E)$. Any cuspidal irreducible representation of $\GL_{N_E}(k_E)$ is up to sign a Deligne-Lusztig character induced from a character $\theta_E$ in general position of the $k_E$-points of an elliptic maximal torus, but equivalently $\theta_E$ is regular.
\end{proof}
\begin{remark}{\rm When $G=\GL_N(F)$ we have also at our disposal the Bushnell-Kutzko construction of supercuspidal irreducible smooth representation in \cite{BKlivre}, which works for any prime number $p$.
By \cite{MY}, this construction coincides with the one in \cite{Yu} when $p$ does not divide the order of the Weyl group of $G$. 

Proposition~\ref{prop:GL} extends to this more general setting as follows.
Let $\pi$ be a supercuspidal irreducible smooth representation of $\GL_N(F)$.   By \cite[Chap.6]{BKlivre}, for $p$ arbitrary, the representation $\pi$ is compactly induced from a representation $\Lambda$ of $E^\times J(\beta,\fA)$, where $J(\beta,\fA)$ is a compact open subgroup of $\GL_N(F)$ and $\Lambda$ is an extension of an irreducible smooth representation $\lambda$ of $J(\beta,\fA)$. Here $E/F$ is a (possibly wildly ramified) finite extension, $\beta$  an element of $\Mat_n(F)$ and $\fA$ an hereditary $\fo_F$-order such that $[\fA,n,0,\beta]$ is a simple stratum. There is a subgroup $J^1(\beta,\fA)$ such that
\begin{equation}
J(\beta,\fA)/J^1(\beta,\fA)\simeq G^{\beta}_{y_\fA,0}/G^{\beta}_{y_\fA,0+},
\end{equation}
where $G^\beta\simeq \GL_{N_E}(E)$ is the centralizer of $E$ in $\GL_N(F)$ and $G^{\beta}_{y_\fA,0}$ denotes the parahoric subgroup of $G^\beta$ associated to a vertex $y_\fA$ in $\cB(G^\beta,E)$ which corresponds to $\fA\cap G^\beta$. Moreover, $\lambda=\infl(\tau_E)\otimes\kappa$, where $\infl(\tau_E)$  is the inflation to $G^{\beta}_{y_\fA,0}$ of an  irreducible cuspidal representation of $G^{\beta}_{y_\fA,0}/G^{\beta}_{y_\fA,0+}\simeq\GL_{N_E}(k_E)$, and $\kappa$ is certain representation of $J(\beta,\fA)$ that extends a Heisenberg representation.
The normalizer of $G^{\beta}_{y_\fA,0}$ coincides with $E^\times G^{\beta}_{y_\fA,0}$. Since  $\tau_E$ is regular, the depth-zero representation $\cInd_{E^\times G^{\beta}_{y_\fA,0}}^{G^{\beta}}(\btau_E)$ is a regular supercuspidal irreducible representation of $G^{\beta}$.
}
\end{remark}

\subsection{Non-singular representations} \label{subsec:ns}
As in \cite[Def.~3.1.1]{Kaletha-nonsingular}, we call a depth-zero character $\vartheta\colon S\to\CC^\times$  \textit{$F$-non-singular} if for every $\alpha\in R_\res(\bL,\bT)$ the character $\vartheta\circ\Nor_{F'|F}\circ\alpha^\vee\colon F^{\prime\times}\to \CC^\times$ has non-trivial restriction to $\fo_{F'}$.
The character $\vartheta$ factors through a character $\theta$ of $\bbT$. We denote the restriction of $\theta$ to $\bbbS^\circ(k_F)$ by $\theta^\circ$. The character $\vartheta$ is called \textit{$k$-non-singular}, if for every $\overline\alpha\in R(\bbbS^\circ,\bbbL_y)$ the character $\theta^\circ\circ\Nor_{k_{F'}|k_F}\circ\overline\alpha^\vee\colon k_{F'}^\times\to \CC^\times$ is non-trivial. As noticed in  \cite[Rem.~3.1.2]{Kaletha-nonsingular}, these notions do not depend on the choice of $F'$.

Let $\vartheta\colon S\to \CC^\times$ be a character. The pair $(\bS,\vartheta)$ is said to be \textit{tame $F$-non-singular elliptic} in $\bL$ (see \cite[Def.~3.4.1]{Kaletha-nonsingular}) if
\begin{itemize}
\item $\bS$ is elliptic (i.e. $\bS$ is anisotropic modulo $\rZ_\bL$) and its splitting extension $E/F$ is tame;
\item  Inside the connected reductive subgroup $\bL^0$ of $\bL$ with maximal torus $\bS$
and root system
\begin{equation}
R_{0+}:= \{\alpha\in R(\bL,\bS)|\vartheta(\Nor_{E|F}(\alpha^\vee(E_{0+}^\times)))=1\},
\end{equation}
the torus $\bS$ is maximally unramified.
\item The character $\vartheta$ is $F$-non-singular with respect to
$\bL^0$.
\end{itemize}
In \cite{Kal-reg, Kaletha-nonsingular}, Kaletha describes how to construct  supercuspidal representations $\sigma_{\bS,\vartheta}$ of $L$ from tame $F$-non-singular elliptic pairs $(\bS,\vartheta)$ in $\bL$. The representation $\sigma_{\bS,\vartheta}$ is obtained in two steps. One starts by unfolding the pair $(\bS,\vartheta)$ into a cuspidal $\bL$-datum $\cD_{\bS,\vartheta}:=(\vv \bL,y,\vv r, \sigma^0,\vv \phi)$. The properties of $\bS$ and $\vartheta$ allow to go to the reductive quotient and use the theory of Deligne-Lusztig cuspidal representations in order to construct $\sigma^0$, the so-called depth-zero part of the datum $\cD_{\bS,\vartheta}$. The second step involves the construction described in \S\ref{subsec:review} to the obtained $\bL$-datum.

\smallskip

Let $\bbbU$ be the unipotent radical of a Borel subgroup of $\bbbL_{y}$. The corresponding Deligne-Lusztig variety is
\begin{equation}
Y_\bbbU^{\bbbL_y}:=\left\{l\bbbU\in \bbbL_{y}/\bbbU\,:\, l^{-1}\Frob(l)\in\bbbU\cdot\Frob(\bbbU)\right\}.
\end{equation}
The group $\bbL_y=\bbbL_y(k_F)$ acts on $Y_\bbbU^{\bbbL_y}$ by the left multiplication and $\bbT=\bbbT(k_F)$ acts on $Y_\bbbU^{\bbbL_y}$ by right multiplication. So these groups act on the $\ell$-adic cohomology group $H^i_c(Y_\bbbU^{\bbbL_y},\overline{\QQ}_{\ell})$. We define
\begin{equation}
H_c^{d_\bbbU}(Y_\bbbU^\bbbL,\overline{\QQ}_{\ell})_{\theta}:=\left\{h\in H^i_c(Y_\bbbU^{\bbbL_y},\overline{\QQ}_{\ell})\,:\,h\cdot s=\theta(s)h\text{ for all $s\in\bbT$}\right\}. 
\end{equation}
Let $\kappa_{(\bbbT,\theta)}^{\bbbL_y}$ be the isomorphism class of the representation $H_c^{d_\bbbU}(Y_\bbbU^\bbbL,\overline{\QQ}_{\ell})_{\theta}$.  The representation $\sigma_{(\bbbT,\theta)}:=\cInd_{L_y}^{L}\infl(\kappa_{(\bbbT,\theta)}^{\bbL_y})$ is supercuspidal but not necessarily irreducible. 

Let $\rN_{\bbL_{[y]}}(\bbbT,\theta)$ denote  the stabilizer of the pair $(\bbbS,\theta)$ in $\bbL_{[y]}$
by the conjugate action, and let $\Irr_\theta(\rN_{\bbL_{[y]}}(\bbbS,\theta))$ be the set of irreducible representations of $\rN_{\bbL_{[y]}}(\bbbT,\theta)$ whose restriction to $\bbS$ is $\theta$-isotypic.
When $\sigma_{(\bbbT,\theta)}$ is reducible, it decomposes as \cite[3.3.3]{Kaletha-nonsingular}
\begin{equation}
    \sigma_{(\bbbT,\theta)}=\sum\limits_{\rho\in \Irr_\theta(\rN_{\bbL_{[y]}}(\bbbT,\theta))}\dim(\rho)\sigma_{(\bbbT,\theta,\rho)}^{\varepsilon},
\end{equation}
where the constituents $\sigma_{(\bbbT,\theta,\rho)}^{\varepsilon}:=\cInd_{L_y}^{L}\infl(\kappa_{(\bbbT,\theta,\rho)}^{\bbL_y,\varepsilon})$, constructed from $\kappa_{(\bbbT,\theta,\rho)}^{\varepsilon}$ as in \cite[Def.~2.7.6]{Kaletha-nonsingular}, are irreducible \textit{non-singular} supercuspidal representations. Here $\varepsilon$ is a fixed coherent splitting of the family of $2$-cocycles $\{\eta_{\Psi,\bbbU}\}$ as in \cite[\S 2.4]{Kaletha-nonsingular}. 
The positive-depth supercuspidals can be described similarly, by applying the construction  of \S\ref{subsec:review} to the representation $\sigma_{(L^0,\bbbT, \phi_{-1})}$ of $L^0$ associated to the pair $(\bbbT, \phi_{-1})$ in \cite[(3.2)]{Kaletha-nonsingular}, where $\phi_{-1}$ is a character of $S$ which is trivial if $\bL^0=\bS$ and otherwise satisfies $\phi_{-1}|_{S_{0+}}=1$.

Let $\sigma\in\Irr_\scusp(L)$ be non-singular, and let $\chi\in\fX_{\nr}(L)$. We observe that $\sigma\otimes\chi$ is also non-singular. 
Then we say that $\fs=[L,\sigma]_G\in\fB(G)$ is non-singular. 
In the following definition, we extend the notions of  regularity and non-singularity to arbitrary irreducible smooth representations of $G$:

\begin{defn} \label{defn:non-singular}
An irreducible smooth representation $\pi$ of $G$ is  \textit{non-singular} if its supercuspidal support is  non-singular, i.e. if $\pi\in\Irr^\fs(G)$ with $\fs=[L,\sigma]_L$ such that $\sigma$ is  non-singular.
\end{defn}

\begin{example}
{\rm We assume that the cardinality $q$ of $k_F$ is odd.  Let $W=\{1,w\}$ denote the Weyl group of $\GL_2$ with respect to a split maximal torus $\widetilde\bT$. We have $\widetilde\bT(k_F)\simeq\FF_q^\times\times\FF_q^\times$. Let $\widetilde\bT_w$ be a torus of type $w$ with respect to $\widetilde\bT$. We have  
\[\widetilde\bT_w(k_F)\simeq\left\{\left(\begin{smallmatrix} \lambda&0\cr 0&\lambda^q\end{smallmatrix}\right)\,:\, \lambda\in\mathbb{F}_{q^2}^\times\right\}\simeq\mathbb{F}_{q^2}^\times.\]
The irreducible cuspidal representations of $\GL_2(k_F)$  are the $-R_{\widetilde\bT_w}^{\GL_2}(\widetilde\omega)$, where $\widetilde\omega$ denote a character of $\mathbb{F}_{q^2}^\times$ such that $\widetilde\omega\ne\widetilde\omega^q$.   
Let $\bT$ be the split maximal torus of $\SL_2$. We have $\bT(k_F)\simeq\mu_{q-1}$. Let $\bT_w$ be a torus of type $w$ with respect to $\bT$. We have
\[\bT_w(\Fq)\simeq\left\{\left(\begin{smallmatrix} \lambda&0\cr 0&\lambda^q\end{smallmatrix}\right)\,:\, \lambda^{q+1}=1\right\}\simeq\mu_{q+1}.\]
Let $\omega$ be a character of $\mu_{q+1}$. Then  we have 
\begin{equation}
-R_{T_w}^{\SL_2}(\omega)=-\Res^{\GL_2(k_F)}_{\SL_2(k_F)}(R_{\widetilde\bT_w}^{\GL_2}(\widetilde\omega)).
\end{equation}
If  $\omega^2\ne 1$, then the representation $\tau:=-R_{T_w}^{\SL_2}(\omega)$ is an irreducible cuspidal representation of $\SL_2(\Fq)$. 
If $\omega=\omega_0$ is the quadratic character of $\mu_{q+1}$, then
 \begin{equation} \label{eqn:reduc}
 -R_{T_w}^{\SL_2}(\omega)=\tau_+ + \tau_,
 \end{equation}
where $\tau_+$ and $\tau_-$ are irreducible cuspidal representations of $\SL_2(k_F)$ of dimension $\frac{q-1}{2}$ (see \cite[Table~2]{DM}).
Let $K:=\SL_2(\fo_F)$ and $K':=uKu^{-1}$ where $u=\left(\begin{smallmatrix} 1&0\cr 0&\varpi_F\end{smallmatrix}\right)$.Then $\{K,K'\}$ is a set of representatives of the maximal parahoric subgroups of $\SL_2(F)$, and the reductive quotients of $K$ and $K'$ are both  isomorphic to  $\SL_2(k_F)$. 
We denote by $\infl(\tau)$ and $\infl(\tau)'$ the inflations of $\tau$ to $K$ and $K'$ respectively, and by $\infl(\tau_\pm)$ and $\infl(\tau_\pm)'$ 
the inflations of $\tau_\pm$ to $K$ and  $K'$ respectively.
Thus, the depth-zero supercuspidal representations of $\SL_2(F)$ are the representations $\cInd_{K}^G(\infl(\tau))$,   $\cInd_{K'}^G(\infl(\tau)')$, which are regular, and the four representations
$\cInd_{K}^G(\infl(\tau_\pm))$, $\cInd_{K'}^G(\infl(\tau_\pm)')$, which are non-regular non-singular.}
\end{example}

\section{Langlands correspondence and Bernstein Center} \label{sec:BC}
\subsection{Generalized Springer correspondence} \label{subsec:GSC}
Let $\cG$ be  a complex (possibly disconnected) reductive group, and let $\cG^\circ$ be the identity component of $\cG$.  If $g\in\cG$, we denote by $(g)_\cG$ its $\cG$-conjugacy class. We denote by $\Unip(\cG)$ the unipotent variety of $\cG$, and  by $\Unip_\enh(\cG)$ the set of $\cG$-conjugacy classes of pairs $(u,\rho)$, with $u\in \cG$ unipotent and  $\rho\in\Irr(A_{\cG}(u))$, where
$A_{\cG}(u):=\rZ_{\cG}(u)/\rZ_{\cG}(u)^\circ$ denotes the component group of the centralizer of $u$ in $\cG$.

Let $\cP_0$ be a parabolic subgroup of $\cG^\circ$ with unipotent radical $\cU$, let $\cM_0$ be a Levi factor of $\cP^\circ$, and let $v$ be a unipotent element in $\cM_0$. The group $\rZ_{\cG^\circ}(u)\times\rZ_{\cM_0}(v)\cU$ acts on the variety
\begin{equation}
\cY_{u,v}:=\left\{y\in\cG^\circ \;:\; y^{-1}uy\in v\cU\right\}
\end{equation}
by $(g,p)\cdot y:= gyp^{-1}$, with $g\in\rZ_{\cG^\circ}(u)$, $p\in\rZ_{\cM_0}(v)\cU$, and $y\in\cY_{u,v}$. 
The group $A_{\cG^\circ}(u)\times A_{\cM_0}(v)$ acts on the set of irreducible components of $\cY_{u,v}$ of maximal dimension, and we denote by $\sigma_{u,v}$ the corresponding permutation representation.

Let $\langle\;,\;\rangle_{A_{\cG^\circ}(u)}$ be the usual scalar product of the space of class functions on the finite group $A_{\cG^\circ}(u)$ with values in $\Qlbar$. Let $\rho^\circ\in\Irr(A_{\cG^\circ}(u))$. The representation $\rho^\circ$ is called \textit{cuspidal}  if
$\langle\rho^\circ,\sigma_{u,v}\rangle_{A_{\cG^\circ}(u)}\ne 0$ for any choice of $v$ implies that $\cP_0=\cG^\circ$.  We denote by $\fB(\Unip_\enh(\cG^\circ))$ the set of $\cG^\circ$-conjugacy classes of triples $\ft^\circ:=(\cM_0,v,\rho^\circ_\cusp)$ such that $\cM_0$ is a Levi subgroup of (a parabolic subgroup of) $\cG^\circ$,  and $(v,\rho^\circ_\cusp)$ is a cuspidal unipotent pair in $\cM_0$, i.e. $v$ is a unipotent element in $\cM_0$, and $\rho^\circ_\cusp$ is a cuspidal irreducible representation of $A_{\cM_0}(v)$.
Let $\ft^\circ\in\fB(\Unip_\enh(\cG^\circ))$ and $W_{\ft^\circ}:=\Nor_{\cG^\circ}(\ft^\circ)/\cM_0$. By \cite[\S9.2]{LuIC}, the group  $W_{\cM_0}:=\Nor_{\cG^\circ}(\cM_0)/\cM_0$ is a Weyl group and $W_{\ft^\circ}\simeq W_{\cM_0}$. 

We consider now the case of the (possibly disconnected) group $\cG$.  Let $u$ be a unipotent element in $\cG^\circ$. We first observe that $\rZ_{\cG}(u)^\circ=\rZ_{\cG^\circ}(u)^\circ$.  (The inclusion  $\rZ_{\cG^\circ}(u)^\circ\subset \rZ_{\cG}(u)^\circ$ is obvious. On the other side, since $\rZ_{\cG}(u)^\circ$ is connected, it is contained in $\cG^\circ$. Hence, $\rZ_{\cG}(u)^\circ$ is contained in $\rZ_{\cG^\circ}(u)$. But, since it is connected, it is contained in  $\rZ_{\cG^\circ}(u)$). It follows that the natural inclusion $\rZ_{\cG^\circ}(u)\subset \rZ_{\cG}(u)$ induces an injection of $A_{\cG^\circ}(u)$ into $A_{\cG}(u)$. A unipotent pair $(u,\rho)$ with $u\in\cG^\circ$ and $\rho\in\Irr(A_{\cG}(u))$ is called \textit{cuspidal} if (one of, or equivalently all) the irreducible constituents of the restriction of $\rho$ to $A_{\cG^\circ}(u)$ are cuspidal. 

We call \textit{quasi-Levi subgroup} of $\cG$  a subgroup $\cM$ of the form
$\cM=\rZ_\cG (\rZ_{\cL}^\circ)$, with $\cL$ a Levi subgroup of $\cG^\circ$. We observe that $\cM^\circ=\cL$. 
Note that a Levi subgroup of $\cG$ is a quasi-Levi subgroup, and in the case when $\cG$ connected, both notions coincide.
The group $\cM$ is said to be {\em cuspidal} if there exists a cuspidal unipotent pair in $\cM$. 

Let $\fL(\cG)$ be a set of representatives for the conjugacy classes of quasi-Levi subgroups of $\cG$, and let $\fL_\cusp(\cG)$ denote the subset of the cuspidal ones. Let $\fB(\Unip_\enh(\cG))$ be the set of $\cG$-conjugacy classes of triples $\ft:=(\cM,v,\rho_\cusp)$, 
where $\cM$ is a cuspidal quasi-Levi subgroup of $\cG$, and $(v,\rho_\cusp)$ is a cuspidal unipotent pair in $\cM$. 
 Let $\ft\in\fB(\Unip_\enh(\cG))$. We set $W_{\ft}:=\Nor_{\cG}(\ft)/\cM$. 
In  \cite[\S4]{AMS1} we constructed a bijection
\begin{equation}  \label{eqn:GSC}
\nu_\cG\colon\bigsqcup_{\ft\in\fB(\Unip_\enh(\cG))} \Irr(\Qlbar[W_{\ft},\kappa_\ft])\,\overset{1-1}{\longrightarrow}\,\Unip_\enh(\cG),
\end{equation}
where $\kappa_{\ft}\colon W_{\ft}/W_{\ft}^\circ
\times W_{\ft}/W_{\ft}^\circ\to\Qlbar^\times$
is a $2$-cocycle and $\Qlbar[W_{\ft},\kappa_\ft]$ is the $\kappa_{\ft}$-twisted group algebra of $W_{\ft}$, which is defined to be the $\Qlbar$-vector 
space $\Qlbar[W_{\ft},\kappa_\ft]$ with basis $\{ T_w \,:\, w \in W_{\ft} \}$ 
and multiplication rules
\begin{equation}\label{eq:1.2}
T_wT_{w'} = \kappa_\ft (w,w') T_{ww'} \quad
\text{for any $w,w' \in W_{\ft}$.}
\end{equation}

We will define a canonical map, called the \textit{cuspidal support} map, 
\begin{equation} \label{eqn:Scmap}
\Scmap_{\cG}\colon \Unip_\enh(\cG)\longrightarrow \fB(\Unip_\enh(\cG)),
\end{equation}
and analyse its fibers. We first consider the $\cG^\circ$-case.
Given a pair $(u,\rho^\circ)$, by \cite[\S~6.2]{LuIC}, there exists a triple $(\cP,\cL,v)$ as above and a cuspidal irreducible representation $\rho_{\cusp}^\circ$ of $A_{\cL}(v)$ such that $\langle \rho^\circ\otimes\widehat\rho_{\cusp}^\circ,\sigma_{u,v}\rangle_{A_{\cG}(u)\times A_\cL(v)}\ne 0$,
where $\widehat\rho_{\cusp}^\circ$ is the dual of $\rho_{\cusp}^\circ$. Moreover $\ft:=(\cP,\cL,v,\rho_{\cusp}^\circ)$ is unique up to $\cG^\circ$-conjugation (see 
\cite[Prop. 6.3]{LuIC}), and is called the \textit{cuspidal support} of $(u,\rho^\circ)$. We define $\Scmap_{\cG^\circ}$ as
\begin{equation} \label{eqn:Scmap0}
\begin{matrix}
\Scmap_{\cG^\circ}\colon& \Unip_\enh(\cG^\circ)&\longrightarrow& \fB(\Unip_\enh(\cG^\circ))\cr
&(u,\rho^\circ)&\mapsto&(u,\rho^\circ)_{\cG^\circ}.
\end{matrix}
\end{equation}
We consider now the $\cG$-case. We set $\cM:=\rZ_{\cG}(\rZ_\cL^\circ)$. Let $\cC_\cL(v)$ denote the $\cL$-conjugacy class of $v$ and $\cE_{\rho_\cusp^\circ}$ the cuspidal $\cL$-equivariant local system on $\cC_\cL(v)$ corresponding to $\rho_\cusp^\circ$.

We denote by $\cC_\cM(v)$ the $\cM$-conjugacy class of $v$ and write $\cX_v:=\cC_\cM(v)\rZ_{\cL}^\circ$.  It is a union of $\cM$-conjugacy classes in $\cL$.  Tensoring $\cE_{\rho_\cusp^\circ}$ with the constant sheaf on $\rZ_{\cL}^\circ$, we obtain a $\cL$-equivariant local system on $\cX_v$ that we still denote by $\cE_{\rho_\cusp^\circ}$.
Next we build a $\cM$-equivariant local system on $\cX_v$ as follows.
Via the map
\begin{equation}
    \cX_v\times \cM\to \cX_v,\quad (x,m)\mapsto m^{-1}xm,
\end{equation}
we pull $\cE_{\rho_\cusp^\circ}$ back to a local system $\widehat\cE_{\rho_\cusp^\circ}$ on $\cX_v\times\cM$. It is $\cM\times\cL$-equivariant for the action
\begin{equation}
(m_1,l)\cdot(x,m):=(m_1xm_1^{-1},m_1ml^{-1}). 
\end{equation}
Since the $\cM^\circ$-action is free, we can divide it out and obtain a $\cM$-equivariant local system $\widetilde\cE_{\rho_\cusp^\circ}$ on $\cX_v\times\cM/\cM^\circ$ the pull-back of which under the natural quotient map is isomorphic to $\widehat\cE_{\rho_\cusp^\circ}$. Let $\pr\colon \cX_v\times\cM/\cM^\circ\to \cX_v$ be the projection on the first coordinate. It is an $\cM$-equivariant fibration, so the direct image $\pr_*\widetilde\cE_{\rho_\cusp^\circ}$ is a $\cM$-equivariant local system on $\cX_v$. By \cite[Prop.~4.5]{AMS1}, there exists a canonical isomorphism 
\begin{equation} \label{eqn:f}
f\colon\End_{\cG}(\pr_*\widetilde\cE_{\rho_\cusp^\circ})\overset{\sim}{\to}\CC[W_\ft,\kappa_\ft],
\end{equation} and  $\End_{\cG}(\pr_*\widetilde\cE_{\rho_\cusp^\circ})$ contains naturally $\End_{\cM}(\pr_*\widetilde\cE_{\rho_\cusp^\circ})$ as a subalgebra. As $\cM/\cM^\circ$ is normal in $W_\ft$, the latter group acts on $\End_{\cM}(\pr_*\widetilde\cE_{\rho_\cusp^\circ})$ by conjugation in $\CC[W_\ft,\kappa_\ft]$. Let $E_{u,\rho}\in\Irr(\End_{\cG}(\pr_*\widetilde\cE_{\rho_\cusp^\circ}))$ which corresponds to 
$(u,\rho)$ by the generalized Springer representation, i.e. $E_{u,\rho}$ is the fiber of $(u,\rho)$ under $\nu_\cG\circ f$, where $\nu_\cG$ is the bijection defined in \eqref{eqn:GSC}. Let $E'$ be a constituent of $E_{u,\rho}$ as an  $\End_{\cM}(\pr_*\widetilde\cE_{\rho_\cusp^\circ})$-representation. We write $\rho_\cusp:=(\pr_*\widetilde\cE_{\rho_\cusp^\circ})_{E'}$. It is a cuspidal irreducible representation  of the group $A_\cM(v)$.
The \textit{cuspidal support} of $(u,\rho)$ is defined to the triple $(\cM,v,\rho_{\cusp})$ (see \cite[(65)]{AMS1}). It is unique up to $\cG$-conjugation.

\subsection{The ordinary Springer correspondence} \label{subsec:Springer}
The bijection $\nu_\cG$ extends to $\cG$ the generalized Springer correspondence for $\cG^\circ$ defined by Lusztig in \cite{LuIC}. The latter itself extends the Springer correspondence $\nu_{\cG^\circ}^\ord$ (sometimes called the \textit{ordinary} Springer correspondence) defined in \cite[\S6]{Sp} and \cite{Lus-Springer}. 

We recall briefly the construction of the map $\nu_{\cG^\circ}^\ord$. 
Let $\cB_u$ be the variety of Borel subgroups of $\cG^\circ$ which contain $u$. According to Spaltenstein, all irreducible components of $\cB_u$ have same dimension, say $e(u)$. The Weyl group $W_{\cG^\circ}$ of $\cG^\circ$ acts on the top cohomology group $H^{2e(u)}(\cB_u,\Qlbar)$
by permuting the irreducible components of $\cB_u$. The  actions of $W_{\cG^\circ}$ and $A_{\cG^\circ}(u)$ commute, and we obtain the irreducible representations of $W_{\cG^\circ}$ by taking the $A_{\cG^\circ}(u)$-isotypic components of $H^{2e(u)}(\cB_u,\Qlbar)$. We define $\Unip_\enh^\ord(\cG^\circ)$ to be the set of pairs $(u,\rho^\circ)$, where $u$ is a unipotent element of $\cG^\circ$ (up to conjugacy) and $\rho^\circ$ is an irreducible representation of $A_{\cG^\circ}(u)$, such that $\Hom_{A_{\cG^\circ}(u)}(\rho^\circ,H^{2e(u)}(\cB_u,\Qlbar))\ne 0$.  
We obtain a bijection, say $\nu_{\cG^\circ}^\ord$, from the set of all irreducible representations of $W_{\cG^\circ}$ to $\Unip_\enh^\ord(G)$.
In particular,  we have
\begin{equation} \label{eqn:Springer}
\text{$(u,1)_{\cG^\circ}\in\nu_{\cG^\circ}^\ord(\Irr(W_{\cG^\circ}))$, for any $u\in\Unip(\cG^\circ)$.}
\end{equation}
Let $(u,\rho)\in\Unip_\enh(\cG)$. By \cite[Theorem~1.2]{AMS1}, we can write $\rho=\rho^\circ\rtimes\eta$,  with $\rho^\circ\in\Irr(A_{\cG^\circ}(u))$. Moreover, where $\rho^\circ$ is uniquely determined by $\rho$ up to $A_{\cG}(u)$-conjugacy.
We extend the definition of $\Unip_\enh^\ord(\cG^\circ)$ to the case of disconnected groups by setting
\begin{equation} \label{eqn:ordinary}
\Unip_\enh^\ord(\cG):=\left\{(u,\rho)\,:\, \text{$(u,\rho^\circ)\in \Unip_\enh^\ord(\cG^\circ)$}\right\}/\cG.
\end{equation}
\begin{defn} \label{defn:ns-triple}
A triple $\ft=(\cM,v,\rho_\cusp)\in\fB(\Unip_\enh(\cG))$ is  called \textit{semisimple} if the unipotent element $v$ is trivial.
\end{defn}
\begin{lem} \label{lem:v=1}
The semisimple triples in $\fB(\Unip_\enh(\cG))$ are those of the form $(\cM,1,\eta)$, where $\cM= \rZ_\cG(\cT)$ with $\cT$  a maximal torus  of $\cG^\circ$, and $\eta\in\Irr(\cM/\cM^\circ)$.
\end{lem}
\begin{proof}
Let $\cT$  be a maximal torus  of $\cG^\circ$  and let  $\eta\in\Irr(\cM/\cM^\circ)$, where $\cM:=\rZ_\cG(\cT)$. By definition the restriction of $\eta$ to $\cM^\circ=\cT$ is the trivial representation of $\cM^\circ$, which is obviously cuspidal. Hence the representation $\eta$ is cuspidal. 
Thus, the triple $(\cM,1,\eta)$ belongs to $\fB(\Unip_\enh(\cG))$. It is semisimple by definition.

Conversely, let $\ft=(\cM,v,\rho_\cusp)\in\fB(\Unip_\enh(\cG))$, where $\cM=\rZ_{\cG}(\rZ_{\cL}^\circ)$ with $\cL$ a Levi subgroup of $\cG^\circ$ and $v$ the trivial element of $\cL$.  By \cite[Prop.~2.8]{LuIC}, the fact that  $(v,\rho_\cusp)$ is cuspidal implies that  $v$ is a distinguished unipotent element in $\cL$, that is, $v$ does not meet the unipotent variety of  any proper Levi subgroup  of $\cL$. Since $v$ is trivial, it is only possible if $\cL$ does not have proper Levi subgroups, that is, $\cL$ is a torus, say $\cT$. 
Thus we have $\rZ_{\cL}=\cT$, which gives $\cM=\rZ_\cG(\cT)$. Since $\rZ_\cM(v)=\cM$, we get $A_\cM(v)=\cM/\cM^\circ$.  Let $\rho\in\Irr(A_{\cG}(u))$. 
By \cite[(51)]{AMS1}, we can write 
\begin{equation}
\rho_\cusp=\rho_\cusp^\circ\rtimes\eta:=\ind_{A_{\cM}(v)_{\rho^\circ_\cusp}}^{A_{\cM}(v)}(\rho^\circ\otimes\eta),
\end{equation}
with $\rho^\circ_\cusp\in\Irr(A_{\cM^\circ}(v))$ and $\eta\in\Irr(\CC[A_{\cM}(v)_{\rho_\cusp^\circ}/A_{\cM^\circ}(v),\kappa'])$,   
where $A_{\cM}(v)_{\rho^\circ_\cusp}$ denotes the stabilizer of $\rho^\circ_\cusp$ in $A_{\cG}(u)$, and  $\kappa'\colon
A_{\cM}(v)_{\rho^\circ_\cusp}/A_{\cM}(v)^\circ\to\Qlbar^\times$ is a certain $2$-cocycle.
Since $v$ is trivial, we have
$A_{\cM^\circ}(v)=\{1\}$. Hence $\rho^\circ_\cusp$ is the trivial representation  (which is cuspidal) of the trivial group $A_{\cM^\circ}(v)$, and we have
$A_{\cM}(v)_{\rho^\circ_\cusp}=A_{\cM}(v)=\cM/\cM^\circ$. It gives
$\rho_\cusp=\eta$ with $\eta\in\Irr(\Qlbar[\cM/\cM^\circ,\kappa'])$, where  $\kappa'\colon
\cM/\cM^\circ\to\Qlbar^\times$ is a $2$-cocycle. But, by \cite[(25)]{AMS1}, the cocycle $\kappa'$ is trivial, since $\rho^\circ_\cusp$ is the trivial representation.
\end{proof}

\begin{prop} \label{prop:key}
The set $\fU_\enh^\ord(\cG)$ is the set of $\cG$-conjugacy classes of pairs $(u,\rho)$ with semisimple cuspidal support.
\end{prop}
\begin{proof}
Let $(u,\rho)_{\cG}\in \fU_\enh(\cG)$. Let $\ft=(\cM,v,\rho_\cusp)$ be the cuspidal support of   $(u,\rho)_\cG$. The generalized Springer correspondence constructed in \cite{LuIC}  reduces to the ordinary Springer correspondence precisely when $\cL$ is a maximal torus of $\cG^\circ$. Thus  $(u,\rho)_\cG$ belongs to $\fU_\enh^\ord(\cG)$ if and only if $\cM= \rZ_{\cG}(\cT)$ with $\cT$ a maximal torus of $\cG$.  But then, since $v$ is a unipotent element  of $\cT$, it has to be trivial. It follows that $A_{\cT}(v)=\{1\}$,  and  hence $\rho_\cusp=1$. So $\ft=(\cM,1,\eta)$, where $\eta\in\Irr(\cM/\cM^\circ)$.  Then the result follows from Lemma~\ref{lem:v=1}.
\end{proof}

\subsection{Enhanced Langlands parameters} \label{subsec:eLp}
Let $G^\vee$ denote the Langlands dual group of $G$, i.e.~the complex Lie group  with root datum dual to that of $G$.
The $L$-group of $G$ is defined to be ${}^LG:=G^\vee\rtimes W_F$.  Let $W_F':=W_F\times\SL_2(\CC)$ be the Weil-Deligne group of $F$.
A \textit{Langlands parameter} (or {$L$-parameter}) for $G$ and its inner twists is a continuous morphism $\varphi\colon W_F'\to {}^LG$ such that 
$\varphi(w)$ is semisimple for each $w\in W_F$, the restriction $\varphi|_{\SL_2(\CC)}$ is a morphism of complex algebraic groups, and $\varphi$ commutes to the projections to $W_F$.
The group $G^\vee$ acts on the set of $L$-parameters. We denote by $\Phi(G)$ the set of $G^\vee$-classes of $G$-relevant $L$-parameters as defined in \cite[\S3]{Bor} (see also \cite[\S2.5]{SZ}).

We define an action of $G^\vee_\sconn$ on $G^\vee$ by setting 
\[h\cdot g:=h'gh^{\prime -1}\quad \text{for $h\in G^\vee_\sconn$ and $g\in G^\vee$, where $p(h)=h'\rZ_{G^\vee}$.}\]
It induces an action of $G^\vee_\sconn$ on ${}^LG$ and we denote by $\rZ_{G^\vee_\sconn}(\varphi)$ the stabilizer in $G^\vee_\sconn$ of $\varphi(W_F')$ for this action.

We attach to each $L$-parameter $\varphi$ several (possibly disconnected) complex reductive groups as follows. 
We set $\rZ_{G^\vee}(\varphi):=\rZ_{G^\vee}(\varphi(W'_F))$ and denote by $\rZ_{G^\vee_\sconn}^1(\varphi)$ the inverse image of $\rZ_{G^\vee}(\varphi)/(\rZ_{G^\vee}(\varphi)\cap \rZ_{G^\vee})$ (viewed as a subgroup of $G^\vee_\ad$) under the quotient map $G^\vee_\sconn\to G^\vee_{\ad}$. We define $\varphi|_{W_F}$ to be the restriction along the embedding of $W_F$ as the first factor of $W'_F$: 
\begin{equation} \label{eqn:restrict}
\begin{matrix}\varphi|_{W_F}\colon& W_F&\to& {}^LG\cr
& w&\mapsto &\varphi(w,1)\end{matrix}.
\end{equation}
We set $\rZ_{G^\vee}(\varphi|_{W_F}):=\rZ_{G^\vee}(\varphi(W_F))$ and denote by $\rZ_{G^\vee}^1(\varphi|_{W_F})$ the inverse image of $\rZ_{G^\vee}(\varphi|_{W_F})/(\rZ_{G^\vee}(\varphi|_{W_F})\cap \rZ_{G^\vee})$ in $G^\vee_\sconn$. We set
\begin{equation} \label{eqn:Gphi}
 \cG_\varphi:=\rZ_{G^\vee}^1(\varphi|_{W_F})
\quad\text{and}\quad
S_\varphi:=\rZ_{G^\vee}^1(\varphi)/(\rZ_{G^\vee}^1(\varphi))^\circ.
\end{equation}
We denote by $u_\varphi$ the image of 
$\left(1,\left(\begin{smallmatrix} 1&1\cr 0&1\end{smallmatrix}\right)\right)$ under $\varphi$.   By \cite[(92)]{AMS1}, we have $u_\varphi\in \cG_\varphi^\circ$ and
$S_\varphi\simeq A_{\cG_\varphi}(u_\varphi)$, where $u_\varphi:=\varphi\left(1,\left(\begin{smallmatrix}
1&1\cr
0&1
\end{smallmatrix}\right)\right)$.

An \textit{enhancement} of $\varphi$ is an irreducible representation $\rho$ of $S_\varphi$. Pairs $(\varphi,\rho)$ are called \textit{enhanced $L$-parameters} (for $G$ and its inner twists). Let $G^\vee$ act on the set of enhanced $L$-parameters  via
\begin{equation} \label{eqn:Gvee-action}
g\cdot (\varphi,\rho) = (g \varphi g^{-1},g \cdot \rho).
\end{equation}
We denote by $\Phi_\enh(G)$ the set of $G^\vee$-conjugacy classes of enhanced $L$-parameters.

An $L$-parameter $\varphi$ is said to be \textit{discrete} if $\varphi(W_F')$ is not contained in any proper Levi subgroup of ${}^LG$.
An enhanced $L$-parameter $(\varphi,\rho)\in\Phi_\enh(G)$ is called \textit{cuspidal} if $\varphi$ is discrete and $(u_\varphi,\rho)$ is a cuspidal unipotent pair in $\cG_\varphi$, as defined in \S\ref{subsec:GSC}.
We denote by $\Phi_{\enh,\cusp}(G)$ the subset of $\Phi_{\enh}(G)$ consisting of  $G^\vee$-conjugacy classes of cuspidal enhanced $L$-parameters.

In \cite[Section 7]{AMS1}, we attached to every enhanced $L$-parameter $(\varphi, \rho)$ its \textit{cuspidal support} $\Sc(\varphi,\rho)$. We recall below how it is defined. 
Let $\ft:=(\cM,v,\rho_\cusp)$ denote the cuspidal support of $(u_\varphi,\rho)$, as defined in \ref{subsec:GSC}.
In particular, $(v,\rho_\cusp)$ is a cuspidal unipotent pair in $\cM$.
Upon conjugating $\varphi$ with a suitable element of $\rZ_{\cG_\varphi^\circ}(u_\varphi)$, we may assume that the identity component $\cL$ of $\cM$ contains $\varphi\left(\left(1, \left(\begin{smallmatrix} z&0\cr 0&z^{-1}
\end{smallmatrix}\right)\right)\right)$ for all $z\in\CC^\times$. 
Recall that by the Jacobson–Morozov theorem (see for example \cite[Theorem~3.7.1]{Chriss-Ginzburg}), any unipotent element $v$ of $\cM$ can be extended to a homomorphism of algebraic groups
\begin{equation} \label{eqn:JM}
j_v\colon \SL_2(\CC)\to \cL\text{ satisfying }j_v\left(\begin{smallmatrix}1&1\\0&1\end{smallmatrix}\right)=v.
\end{equation}
By \cite[Prop.~3.7.3]{Chriss-Ginzburg}, the homomorphism $j_{v}$ is uniquely determined up to conjugation by an element of the unipotent radical of $\rZ_{\cL}(v)$. In particular, $j_1$ is the identity map. A homomorphism $j_v$ satisfying \eqref{eqn:JM} will be said to be \textit{adapted to $\varphi$}. 

By \cite[Lem.~7.6]{AMS1}, up to $G^\vee$-conjugacy, there exists a unique homomorphism $j_{v}\colon \SL_2(\CC)\to \cL$ which is adapted to $\varphi$, and moreover, the cocharacter
\begin{equation} \label{eqn:cocharacter}
\chi_{\varphi,v}\colon z\mapsto \varphi\left(1, \left(\begin{smallmatrix} z&0\cr 0&z^{-1}
\end{smallmatrix}\right)\right)\cdot j_v\left(\begin{smallmatrix} z^{-1}&0\cr 0&z
\end{smallmatrix}\right)
\end{equation}
has image in $\rZ_{\cL}^\circ$.
We define an $L$-parameter $\varphi_{v} \colon W_F \times\SL_2(\CC) \to\rZ_{G^\vee}(\rZ_{\cL}^\circ)$ by
\[\varphi_v(w,x):= \varphi(w,1)\cdot\chi_{\varphi,v}(|\!|w|\!|^{1/2})\cdot j_v(x)\quad\text{for any $w\in W_F$ and any $x\in \SL_2(\CC)$.}\]
The \textit{cuspidal support} of $(\varphi,\rho)$ is defined to be
\begin{equation} \label{eqn:Sc Galois side}
\Sc(\varphi,\rho):=(\rZ_{{}^LG}(\rZ_{\cL}^\circ),(\varphi_{v},\rho_\cusp))_{G^\vee}.
\end{equation}
By \cite[Prop.~7.3]{AMS1}, we can always represent $(\rZ_{{}^LG}(\rZ_{\cL}^\circ),\varphi_{L},\rho_\cusp)_{G^\vee}$ by  a triple of the form   $(L^\vee\rtimes W_F,\varphi_{L},\rho_L)$, where $L$  is a Levi subgroup of $G$ and $(\varphi_{L},\rho_L)$ is a cuspidal enhanced $L$-parameter for $L$.
The cuspidal support of  $(\varphi,\rho)\in\Phi_\enh(G)$ is said to be \textit{semisimple} if $v=1$.

\begin{defn} \label{defn:ss}
An enhanced $L$-parameter $(\varphi,\rho)\in\Phi_\enh(G)$ is called \textit{non-singular} if its  cuspidal support  is semisimple.
\end{defn}
 We recall that an $L$-parameter $\varphi_L\colon W'_F\to{}^LL$ is said to be \textit{supercuspidal} if it is discrete and has trivial restriction to $\SL_2(\CC)$.
Thus, an enhanced $L$-parameter $(\varphi,\rho)\in\Phi_\enh(G)$ is non-singular if and only if its cuspidal support is of the form $(L^\vee\rtimes W_F,\varphi_{L},\rho_L)_{G^\vee}$, with $\varphi_L$  a supercuspidal $L$-parameter for a Levi subgroup $L$ of $G$.

\begin{thm} \label{thm:Springer}
An  enhanced $L$-parameter $(\varphi,\rho)\in\Phi_\enh(G)$ is 
non-singular if and only if $(u_\varphi,\rho)_{\cG_\varphi}\in\fU^\ord_\enh(\cG_\varphi)$.
\end{thm}
\begin{proof} It follows immediately from Proposition~\ref{prop:key}.
\end{proof}

\subsection{A Bernstein-type decomposition on the Galois side}\label{subsec:Bdecomp}
Let $\bL$ be a Levi subgroup defined over $F$ of a parabolic subgroup of $\bG$ defined over $F$ and let $L^\vee$ denote the $\CC$-points of the Langlands dual group of $\bL$, 
We define
\begin{equation} \label{eqn:Xnrdual}
\fX_\nr({}^LL):=(\rZ_{L^\vee}^{I_F})_{W_F}^\circ.
\end{equation}
The group $\fX_\nr({}^LL)$ is naturally isomorphic to the group $\fX_\nr(L)$ (see \cite[p.16]{Haines}). We denote the isomorphism $\fX_\nr(L)\xrightarrow{\sim} \fX_\nr({}^LL)$ by $\chi\mapsto\chi^\vee$.  We define an action of $\fX_\nr({}^LL)$ on $\Phi_\enh(L)$ as follows. 
Given $(\varphi,\rho)\in\Phi_\enh(L)$ and $\xi\in \fX_\nr({}^LL)$, we define $(\xi\varphi,\rho)\in\Phi_\enh(L)$ by 
$\xi\varphi:=\varphi$ on $I_F\times\SL_2(\CC)$ and  $(\xi\varphi)(\Frob_F):=\tilde \xi \varphi(\Frob_F)$. Here $\tilde \xi\in \rZ_{L^\vee\rtimes I_F}^\circ$ represents $z$. Let $\fs^\vee:=[L^\vee,\varphi_\cus,\rho_\cus]_{G^\vee}$ be the $G^\vee$-conjugacy class of the orbit of $(\varphi_\cus,\rho_\cus)\in \Phi_\enh^\cusp(L)$ under the action of 
$\fX_\nr(L^\vee)$. Let $\fB^\vee(G)$ be the set of such $\fs^\vee$. 

By \cite[(115)]{AMS1}, the set $\Phi_\enh(G)$ is partitioned into \textit{series \`a la Bernstein} as
\begin{equation} \label{eqn:decPhi_e}
\Phi_\enh(G)=\prod_{\fs^\vee\in\fB(G^\vee)}\Phi(G)^{\fs^\vee},
\end{equation}
where for each $\fs^\vee$, the subset $\Phi_\enh(G)^{\fs^\vee}$
is defined to be the fiber of $\fs^\vee$ under the supercuspidal map $\Sc$ defined in \eqref{eqn:Sc Galois side}. 

For any $\fs=[L,\sigma]_G\in\fB(G)$, it is expected (see \cite[Conj.~2]{AMS1}), and known for $p$-adic general linear groups and for pure inner forms of quasi-split $p$-adic classical  groups (by \cite{Moussaoui-Bernstein-center, AMS4}), when $p\ne 2,3$,  for the exceptional group of type $\rG_2$ (by \cite{AX-Hecke,AX-G2}), and  for $G$ arbitrary and  $L$ a torus (by \cite{ABPS-disc, Sol-ppal}), that some suitable notion of  $\LLC$ for supercuspidal representations of $L$, denoted by $\sigma'\mapsto(\varphi_{\sigma'},\rho_{\sigma'})$, induces a bijection:
\begin{equation} \label{Bernstein-block-isom-intro}
   \begin{matrix}\cL^\fs_G\colon& \Irr^{\fs}(G)&\rightarrow&\Phi_\enh^{\fs^{\vee}}(G)\cr
  &\pi&\mapsto&(\varphi_\pi,\rho_\pi) 
  \end{matrix}
\end{equation}
where $\fs^{\vee}=[L^{\vee},\varphi_{\sigma},\rho_{\sigma}]_{G^\vee}$, and hence, by using \eqref{eqn:Bernstein decomposition} and \eqref{eqn:decPhi_e}, a bijection
\begin{equation} \label{eqn:LG}
\begin{matrix}
\cL_G\colon&\Irr(G)&\rightarrow&\Phi_\enh(G)\cr
&\pi&\mapsto&(\varphi_\pi,\rho_\pi)
 \end{matrix}.
 \end{equation}

Let $\fX_{\nr}(L,\sigma)$ denote the finite subgroup of $\fX_{\nr}(L)$ defined as
\begin{equation} 
\fX_{\nr}(L,\sigma):=\left\{\chi\in \fX_{\nr}(L)\;:\;\chi\otimes\sigma=\sigma\right\}.
\end{equation}
Let $\natural$ and ${}^L\natural$ denote the collections of $2$-cocycles, occurring respectively in \cite{Solleveld-endomorphism-algebra} and  \cite{AMS1}. In many  cases these cocycles are known to be both trivial (for instance for classical groups, as proved in \cite{AMS4}).

For every Levi subgroup $L$ of $G$, there is a canonical isomorphism  $w\mapsto w^\vee$ between $W_G(L):=\Nor_G(L)/L$ and $W_{G^\vee}(L^\vee):=\Nor_{G^\vee}(L^\vee)/L^\vee$, see \cite[Prop.~3.1]{ABPS-CM}. For $w\in W_G(L)$ and $\sigma\in\Irr_\scusp(L)$, let ${}^w\sigma$ denote the representation of $L$ defined by
${}^w\sigma(\ell):=\sigma(w^{ -1}\ell w)$. If the representation $\sigma$ is non-singular, then the representation ${}^w\sigma$ is also non-singular.

\begin{remark} \label{rem:properties}  {\rm
If the following properties are satisfied for every $\sigma'\in\fs=[L,\sigma]_G$:
\begin{equation} \label{eqn:functK}
{}^{w^\vee}(\varphi_{\sigma'},\rho_{\sigma'})=(\varphi_{{}^w\sigma'}, \rho_{{}^w\sigma'})\quad\text{for any $w\in W(L)$.}
\end{equation}
\begin{equation} \label{eqn:fcocycles}
{}^L\natural_{\chi^\vee}=\natural_\chi\quad\text{for any $\chi\in \fX_\nr(L)/\fX_{\nr}(L,\sigma)$,}
\end{equation}
then $\sigma'\mapsto(\varphi_{\sigma'},\rho_{\sigma'})$ induces a bijection $\cL^\fs_G$ from $\Irr^{\fs}(G)$ to $\Phi_\enh^{\fs^{\vee}}(G)$, where  $\fs^{\vee}=[L^{\vee},\varphi_{\sigma},\rho_{\sigma}]_{G^\vee}$ (see  \cite{AX-Hecke}).
}
\end{remark}

Moreover, as formulated in \cite[Conjecture-7.8]{AMS1}, the following diagram is expected to be commutative:
\begin{equation} \label{eqn:diagram}
\begin{tikzcd}
\Irr(G)\arrow[]{d}[swap]{\Sc}\arrow[]{r}{\cL_G}[swap]{} &\Phi_\enh(G)\arrow[]{d}{\Sc}\\
\bigsqcup_{L\in\fL(G)}\Irr_{\superc}(L)/W_G(L)\arrow[]{r}{\LLC}[swap]{1\text{-}1}& \bigsqcup_{L\in\fL(G)}\Phi_{\enh,\cusp}(L)/W_G(L).
\end{tikzcd},
\end{equation}
where $\fL(G)$ is a set of representatives for the conjugacy classes of Levi subgroups of $G$.
The commutativity of the diagram \eqref{eqn:diagram} was established   for general linear groups and for classical groups in \cite{Moussaoui-Bernstein-center, AMS4}, for $\rG_2$ in \cite{AX-G2}, and  for $G$ arbitrary and  $L$ a torus in \cite{ABPS-disc, Sol-ppal}.

\subsection{The $\LLC$ for non-singular representations} \label{subsec:LLCns}
Let $L$ be Levi subgroup of $G$.  
In \cite[Def.~4.1.2]{Kaletha-nonsingular},  a supercuspidal $L$-parameter $\varphi_L$ for $L$ is said to be \textit{torally wild} if $\varphi(P_F)$ is contained in a maximal torus of $L^\vee$. In particular, if $\varphi_L$ is torally wild then $(\rZ_{L^\vee}(\varphi(I_F)))^\circ$ is abelian. 

In \cite{Kaletha-nonsingular}, when the group $\bL$ splits over a tamely ramified  extension and $p$ is very good for $\bL$, Kaletha has constructed the $L$-packets associated to torally wild supercuspidal $L$-parameters and has attached  to each non-singular irreducible supercuspidal representation $\sigma$ of $L$ a cuspidal enhanced $L$-parameter $(\varphi_\sigma,\rho_\sigma)$ such that $\varphi_\sigma$ is a torally wild supercuspidal $L$-parameter for $L$. When $G=\GL_N(F)$ and $p$ is odd, as proved in \cite{OT}, the map $\sigma\mapsto (\varphi_\sigma,\rho_\sigma)$ coincides with the local Langlands correspondence defined by Harris and Taylor in \cite{HT} and Bushnell et Henniart in \cite{BH1,BH2}.

The connection index for $\bL$ dividing the order of Weyl group of $\bL$ (see for instance \cite[VI.2, Prop.7]{Bourbaki}), if $p$ does not divide the order of the Weyl group of $\bL$, then $p$ is very good for $\bL$. When $\bL$ splits over a tamely ramified  extension and $p$ does not divide the order of the Weyl group of $\bL$, all supercuspidal $L$-parameters are torally wild, and hence, in this case, the $L$-packets associated to all supercuspidal $L$-parameters for $L$ are constructed in \cite{Kaletha-nonsingular}: these are exactly the $L$-packets that contain non-singular supercuspidal representations of $L$.

Let $\Inn(\bG)$ be the group of inner automorphisms of $\bG$. Recall that given an algebraic group $\bG'$ over $F$, an isomorphism $\xi\colon \bG' \to \bG$ defined over $F_s$ determines a $1$-cocycle
\begin{equation}\label{eqn:iso}
z_\xi \colon  \begin{array}{ccc} \Gamma_F & \to & \Aut (\bG) \\
\gamma & \mapsto & \xi \gamma \xi^{-1} \gamma^{-1} .                 
\end{array}
\end{equation}
An \textit{inner twist} of $G$ consists of a pair $G'_\xi:=(G',\xi)$, where $G'=\bG'(F)$ for some connected reductive $F$-group $\bG'$, and $\xi\colon\bG'\isom\bG$ is an isomorphism of algebraic groups defined over $F_\sep$ such that $\text{im } (\gamma_\xi) \subset \Inn(\bG)$. Two inner twists $(G', \xi), (G'_1, \xi_1)$ of $G$ are {\em equivalent} if there exists $f \in \Inn(\bG)$ such that
\begin{equation} \label{eq:cohomologous}
z_\xi (\gamma) = f^{-1} z_{\xi_1} (\gamma) \; \gamma f \gamma^{-1}
\quad \forall \gamma \in \Gamma_F .
\end{equation}
We denote the set of equivalence classes of inner twists of $G$ by $\InnT(G)$.

An inner twist of $G$ is the same thing as an inner twist of the unique quasi-split inner form $G^* = \bG^* (F)$ of $G$.
Thus the equivalence classes of inner twists of $G$ are parametrized by the Galois cohomology group $H^1(F,\Inn(\bG^*))$.

Kottwitz \cite[Theorem 6.4]{Ko} found a natural group isomorphism
\begin{equation}
H^1(F,\Inn(\bG^*))\isom\Irr((\rZ_{G^\vee_\sconn})^{W_F}).
\end{equation}
(When F has positive characteristic, see \cite[Theorem 2.1]{Th}.) In this way every
inner twist of $G$ is associated to a unique character of $ (\rZ_{G^\vee_\sconn})^{W_F}$.

The basic version of the local Langlands conjecture states that there exists a surjective finite-to-one
map from the set of equivalence classes of irreducible admissible representations of $G$
to the set of $G^\vee$-conjugacy classes of relevant $L$-parameters $\varphi\colon W'_F\to {}^LG$.  The fiber over $\varphi$ is called an 
\textit{$L$-packet}. we denote it by $\Pi_\varphi(G)$.
Assuming the validity of the basic local Langlands conjecture, we can define for
each Langlands parameter $\varphi$ the “compound $L$-packet”  $\Pi_\varphi$ as
\begin{equation}
\Pi_\varphi:=\bigcup_{G'_\xi\in\InnT(G)}\Pi_\varphi(G'_\xi).
\end{equation}
The elements of $\Pi_\varphi$ are expected to be in bijection with $\Irr(S_\varphi)$, where $S_\varphi$ is the group introduced in \eqref{eqn:Gphi} (see \cite{Ar} and \cite[\S4.6]{Kal-global-rigid}).
\begin{theorem} \label{thm:main}
We suppose that $\bG$ splits over a tamely ramified  extension, that $p$ does not divide the order of the Weyl group of $\bG$, and that, for every Levi subgroup $L$ of $G$, the Langlands correspondence
$\cL^{\Kal}_L\colon \sigma\mapsto (\varphi_{\sigma},\rho_{\sigma})$ defined by Kaletha for the non-singular supercuspidal representations $\sigma$ of $L$ satisfies \eqref{eqn:functK}, and \eqref{eqn:fcocycles}, and that the induced maps $\cL^\fs_G$, where $\fs=[L,\sigma]_G$, satisfy \eqref{eqn:diagram}.

Let $\varphi\colon W'_F\to {}^LG$ be an $L$-parameter. Then the equivalence classes of the irreducible non-singular representations in the compound  $L$-packet $\Pi_\varphi$ are in bijection with 
$\Unip_\enh^\ord(\cG_\varphi)$. In particular, the $L$-packet $\Pi_\varphi(G^*)$ contains at least one irreducible non-singular representation. 
\end{theorem}
\begin{proof}
Let $\fs=[L,\sigma]_G$  non-singular. Since $\bG$ splits over a tamely ramified  extension, $\bL$ also splits over a tamely ramified  extension, and, since $p$ does not divide the order of the Weyl group of $\bG$, it does not divide the order of the Weyl group of $\bL$.
Since we have assume that $\cL^{\Kal}_L$ satisfies \eqref{eqn:functK}and \eqref{eqn:fcocycles}, it follows from Remark~\ref{rem:properties} that we have a bijection 
\[\cL^\fs_G\colon\Irr^{\fs}(G)\to\Phi_\enh^{\fs^{\vee}}(G),\] where $\fs^{\vee}=[L^{\vee},\varphi_{\sigma},\rho_{\sigma}]_{G^\vee}$, and, 
by hypothesis, the map $\cL_G^\fs$ satisfies \eqref{eqn:diagram}. 

Let $\varphi\colon W'_F\to {}^LG$ be an $L$-parameter and let $\rho\in\Irr(S_\varphi)$. Let  $(L^\vee\rtimes W_F,\varphi_{L},\rho_L)_{G^\vee}$ denote the cuspidal support of  the enhanced $L$-parameter $(\varphi,\rho)$.  By Theorem~\ref{thm:Springer},  we have $(u_\varphi,\rho)_{\cG_\varphi}\in\fU^\ord_\enh(\cG_\varphi)$ if and only if  $(\varphi,\rho)$  is non-singular. It follows from Definition~\ref{defn:ss} that $(\varphi,\rho)$  is non-singular if and only if  the $L$-parameter $\varphi_{L},$ is supercuspidal. But, by \cite{Kaletha-nonsingular}, the  supercuspidal  $L$-parameters for $L$ are exactly the ones which correspond to non-singular supercuspidal irreducible representations of $L$ by the Langlands correspondence $\cL^{\Kal}_L$. 

Since, by \eqref{eqn:Springer}, we have $(u_\varphi,1)_{\cG_\varphi}\in\fU^\ord_\enh(\cG_\varphi)$, there exists a non-singular irreducible representation $\pi$ of the quasi-split group $G^*$ such  that $\cL^\fs(\pi)=(\varphi,1)$.  
\end{proof}

\end{document}